\documentclass{amsart}

\usepackage{amsmath}
\usepackage{amssymb}
\usepackage{amsfonts}
\usepackage{amsthm}

\newcommand{\NN}{\mathbb{N}}
\newcommand{\RR}{\mathbb{R}}

\newcommand{\boundary}{\partial}
\newcommand{\efe}{\mathcal{F}_\beta}

\numberwithin{equation}{section}

\newtheorem{thm}{Theorem}[section]

\newtheorem{lem}[thm]{Lemma}

\newtheorem{Rema}[thm]{Remark}

\begin{document}

\title[Interior $L^p$ - estimates for elliptic and parabolic]{Interior $L^p$ - estimates for elliptic and parabolic Schr\"{o}dinger type operators and local $A_p$ -weights}

\author{Isolda Cardoso -- Pablo Viola -- Beatriz Viviani}

\address{
  Departamento de Matem\'{a}tica  \\
  Fac. de Cs. Exactas, Ingenier\'{i}a y Agrimensura  \\
 Universidad Nacional de Rosario \\
  Pellegrini 250, 2000 Rosario
 }
 \email{isolda@fceia.unr.edu.ar }

 \address{
 Facultad de Ciencias Exactas  \\
  Universidad Nacional del Centro de la Provincia de Buenos Aires  \\
 Pinto 399, 7000 Tandil
 }
 \email{pviola@exa.unicen.edu.ar}

\address{
  Instituto de Matem\'{a}tica Aplicada del Litoral \\
 CONICET- Universidad Nacional del Litoral \\
 IMAL-CCT CONICET Santa FE \\
 Colectora Ruta Nac. No 168, Paraje El Pozo
3000 Santa Fe
}
\email{vivani@santafe-conicet.gov.ar}

\begin{abstract}

 Let $\Omega$ be a non-empty open proper and connected 
 subset of $\mathbb R^{n}$. 
 Consider the elliptic Schr\"{o}dinger type operator 
 $L_{E}u=$
 $A_{E}u+Vu=$
 $-\Sigma_{ij}a_{ij}(x)$ 
 $u_ {x_i x_j}+Vu$
 in $\Omega$, 
   and the linear parabolic operator 
   $L_{P}u=A_{P}u+Vu=$
   $u_{t}-\Sigma a_{ij}(x,t)u_{x_{i}x_{j}}+Vu$ 
   in $\Omega_{T}=\Omega\times ( 0,T )$, 
   where the coefficients $a_{ij}\in VMO$ and the potential $V$ satisfies a reverse-H\"{o}lder condition. 
   The aim of this paper is to obtain a priori estimates for the operators $L_{E}$ and $L_{P}$ in weighted Sobolev spaces involving the distance to the boundary and weights in a local- $A_{p}$ class.

\

\noindent
{\bf Mathematics Subject Classification (2010):} Primary: 35J10; Secondary: 35B45; 42B35.
\end{abstract}

 \maketitle

\section{Introduction}
\label{Intro}

  Let $\Omega$ be a non-empty open proper and connected subset of $\RR^{n}$. We are going to consider the following two operators: the elliptic Schr\"{o}dinger type operator

\begin{equation*}
L_{E}u=A_{E}u+Vu=-\sum_{ij}a_{ij}(x)u_ {x_i x_j}+Vu
\end{equation*}
in $\Omega$, and the linear parabolic operator

\begin{equation*}
L_{P}u=A_{P}u+Vu=u_{t}-\sum a_{ij}(x'                ,t)u_{x_{i}x_{j}}+Vu
\end{equation*}
 in $\Omega_{T}=\Omega\times ( 0,T )$, with $T>0$, under the following assumptions:
 \begin{itemize}
			\item[(1)] $a_{ij}=a_{ji}$, and
\begin{equation*}
\frac{1}{C} |\xi|^{2} \leq \sum _{ij} a_{ij}(.)\xi_{i}\bar{\xi_{j}}\leq C|\xi|^{2}
\end{equation*}
 for a.e. $x\in \Omega$ or $x=(x',t)\in \Omega_{T}$, respectively;
			\item[(2)] $a_{ij}\in L^{\infty}\cap VMO(\RR^{n})$. Here we have the space of functions of vanishing mean oscillation defined as

\begin{equation*}
VMO(\RR^{n})=\big\{g\in BMO(\RR^{n}): \eta(r)\to 0, r\to 0^{+}\big\},
\end{equation*}
 where
\begin{equation*}
\eta(r)=\sup _{\rho\leq r} \sup _{x\in\RR^{n}}\Bigg(\dfrac{1}{|B_{\rho}(x)|}\int _{B_{\rho}(x)}\big|g(y)-g_{B_{\rho}}\big|dy\Bigg).
\end{equation*}
Here $g_{B_\rho }=|B(\rho(x) )|^{-1}\int_{B_\rho (x)} g(y)\, dy$. The parabolic $VMO(\RR^{n+1})$ is defined in the same way, except this time we take the supremum over the parabolic balls (see section \ref{prelim:parabolicsetting});
			\item[(3)] The potential $V\geq 0$ satisfies a reverse H\"{o}lder condition of order $q$, shortly $V\in RH_{q}$, which means that
\begin{equation}\label{RHq}
  \Big(\dfrac{1}{|B|}\int _{B} V^q dx\Big)^{1/q} \leq \dfrac{1}{|B|}\int _{B} Vdx,
\end{equation}
where the ball $B$ is in $\RR^{n}$.
\end{itemize}

   Sometimes we  will use $A$ for either the operators $A_{E}$ or $A_{P}$, and $\Lambda$ for either the subset $\Omega$ or $\Omega_{T}$.

   When the coefficients $a_{ij}$ are at least uniformly continuous, existence and uniqueness results together with a-priori $W^{2,p}$ estimates are well known (see e.g. \cite{GT}). The theory for operators with discontinuous coefficients, in the sense of $VMO$, goes back to the 90's with the works of Chiarenza-Frasca-Longo in \cite{CFL1} and \cite{CFL2} for  elliptic operators and  Bramanti-Cerutti in \cite{BC} for the parabolic case. Since then, many authors have considered this problem in different situations and contexts.
   The Schr\"{o}dinger operator when $A$ is the Laplacian and the potential $V$ satisfies the reverse-H\"{o}lder condition (3), was studied by Shen in \cite{Sh2} and  related results when $V(x) =|x|$ (Hermite operator) have been proved by Thangavelu in \cite{T}. For the elliptic type Schr\"{o}dinger operator under consideration, a global $W^{2,p}(\RR^{n})$ estimate and the existence and uniqueness results deduced from them were obtained in \cite{BBHV}.
   We are interested in obtaining a priori interior estimates in weighted Sobolev spaces for the operator $L$, where $L$ is either the elliptic Schr\"{o}dinger type operator $L_{E}$ or the parabolic operator $L_{P}$, defined in a non necessarily bounded domain. We follow the strategy adopted in \cite{BBHV}. First we get a weighted version of the a priori estimates obtained in \cite{CFL1} and in \cite{BC} for the principal operator $A_E$ and $ A_P$ respectively. Thanks to these estimates we are reduced to prove a weighted $L^p$ bound on $Vu$ in terms on $Lu$. Then, we give a representation formula for $Vu$ by means of the fundamental solution of a constant coefficient operator of the type $A_0 + V$, for which a global estimate was proved by Dziubanski in \cite{D} for $L_{E}$ and by Kurata in \cite{K} for $L_{P}$. These representation formulas involve suitable integral operators with positive kernel, applied to $Lu$, and their positive conmutators, applied to the second order derivatives of $u$.

   In order to prove that these operators are bounded on  weighted $L^p$, we use local maximal functions, $M_{\text{loc}}f$ (see section \ref{prelim}), defined in a proper open set imbedded in a metric space. This maximal operator and the  classes of weight involved $A_{p,\text{loc}}$ (see below), were first studied by Nowak and Stempak in \cite{NS} when $\Omega=(0,\infty)$ and by Lin and Stempak in \cite{LS} for $\Omega= \RR^{n}\setminus \{0\}$. In a general setting, that is in metric spaces, this maximal operator and the corresponding classes of weights were considered by Harboure, Salinas and Viviani in \cite{HSV} and by Lin, Stempak and Wan in \cite{LSW}.

   We consider the local weights class $A_{p,\text{loc}}$ defined as follows: let $(X,d)$ be a metric space and let $\Lambda$ be a nonempty open proper subset of $X$, if $0< \beta < 1$ we define the family of balls

\begin{equation*}
F_{\beta} =\big\{ B=B(x_{B},r_{B}): x_{B}\in\Gamma,  r_{B}< \beta d(x_{B},\Lambda^{C}) \big\},
\end{equation*}
  where $d(x_{B},\Lambda^{C})$ denotes the distance from the center $x_B$ of the ball $B$ to the complementary set of $\Lambda$.
Given a Borel measure $\mu$ defined on $\Lambda$, for $1< p< \infty$, we define
\begin{equation} \label{apbeta}
\mbox{ $w\in A^{\beta}_{p,\text{loc}}(\Lambda)$\ \  iff \ \ }
 \sup _{B\in\mathcal{F}_{\beta} }\frac{1}{\mu(B)}
 \Big(\int _{B} w d\mu\Big)^{1/p}
 \Big(\int _{B} w^{-p/p'}d\mu\Big)^{1/p'} < \infty.
\end{equation}
We remark that the classes $ A^{\beta}_{p,\text{loc}}(\Lambda)$ are independent of $\beta$, as was shown in \cite{HSV}. In view of this fact, we shall refer to theses weights as
$ A_{p,\text{loc}}(\Lambda)$.\
   We also consider the following weighted Sobolev spaces,
   defined in $\RR^{n}$ and $\RR^{n+1}$, respectively:

\begin{equation*}
W^{2,p}_{\delta ,w}(\Omega )=
   \Big\{u \in L^{1}_{\text{loc}}(\Omega ): \|u\|_{W^{2,p}_{\delta ,w}(\Omega )}=\sum _{|\gamma|\leq 2} \|\delta ^{|\gamma |} D^{\gamma} u\|_{L^{p}_{w}(\Omega )}< \infty\Big\},
\end{equation*}
 and
\begin{equation*}
W^{2,p}_{\delta ,w}(\Omega_{T} \! )
  \!  =   \!
 \Big\{  \! u   \!   \in   \!   
   L^{1}_{\text{loc}}  \!  (\Omega_{T} )
  \!    :   \!
    \|u\|_{W^{2,p}_{\delta ,w}(\Omega_{T} )}
   \!\!  = \!\!\!
   \sum _{|\gamma |\leq 2} 
   \!\!
      \|\delta ^{|\gamma|} D_{x}^{\gamma} u\|_{L^{p}_{w}(\Omega_{T} )} + \|\delta^{2}D_{t}u\|_{L^{p}_{w}(\Omega_{T})}
        \!\! < \!  \infty   \!
      \Big\}   \!   ,
\end{equation*}
where $\delta(x)=\min \{ 1 , d(x,\Lambda^{C})\}$, with either $\Lambda=\Omega$ or $\Omega_{T}$, and $d$ denotes the corresponding distance.

   We will prove the following results:

 \begin{thm}
     \label{thm:principal}
Let $\Omega$ be a nonempty, proper, open and connected subset of $\RR^{n}$. Let
      $p\in(1,q]$ and $w\in A_{p,\text{loc}}(\Omega )$.
     If $u\in W^{2,p}_{\delta ,w}(\Omega )$ is a solution of
     \begin{equation*}Lu = Au+Vu =-\sum_{i,j}a_{ij}u_{x_ix_j}+Vu=f\qquad\text{in $\Omega $},\end{equation*}
     under the assumptions (1), (2) and (3), then
     \begin{equation*}\|u\|_{W^{2,p}_{\delta ,w}(\Omega )} + \|\delta ^2Vu\|_{L^p_{w}(\Omega )}\leq C\big[\|\delta ^2 f\|_{L^p_w(\Omega )}+\|u\|_{L^p_w(\Omega )}\big],\end{equation*}
     where $\delta (x)=\min\{1,d(x,\Omega^{C})\}, \,x\in \RR^{n}$.
  \end{thm}

   The parabolic version of this theorem goes as follows:

 \begin{thm}
     \label{thm:principalP}
Let $\Omega$ be a nonempty, proper, open and connected subset of $\RR^{n}$. For $T>0$ define $\Omega_{T}=\Omega\times \big(0,T\big)$. Let
      $p\in(1,q]$ and $w\in A_{p,\text{loc}}(\Omega_{T} )$.
     If $u\in W^{2,p}_{\delta ,w}(\Omega_{T})$ is a solution of
     \begin{equation*}Lu = Au+Vu =u_{t}-\sum_{i,j}a_{ij}u_{x_i x_j}+Vu=f\qquad\text{in $\Omega_{T} $},\end{equation*}
     under the assumptions (1), (2) and (3), then
     \begin{equation*}\|u\|_{W^{2,p}_{\delta ,w}(\Omega_{T} )} + \|\delta ^2Vu\|_{L^p_{w}(\Omega_{T} )}\leq C\big[\|\delta ^2 f\|_{L^p_w(\Omega_{T} )}+\|u\|_{L^p_w(\Omega_{T} )}\big],\end{equation*}
     where $\delta (x',t)=\min\{1,d((x',t),\Omega^{C}_{T} )\}$ .
  \end{thm}

We note that, as it is easy to check, $w(x)= \delta^\alpha(x)$ belongs to
 $A_{p,\text{loc}}$ for any exponent $\alpha\in\mathbb{R}$. Therefore the data function $f$ appearing on the right hand side of Theorem \ref{thm:principal} and Theorem \ref{thm:principalP}  could increase polynomially when approaching the boundary of $\Omega$ or $\Omega_{T}$ and still we might have some control for the derivatives of the solution up to the order $2$.

   The paper is organized as follows: in Section \ref{prelim} we put together the preliminary definitions and results, and prove some useful lemmas; in Section \ref{previous} we prove some results that will build the proof of the Main Theorem for the operator $L_{E}$,
   and in Section \ref{previousP} we show similar results for the operator $L_{P}$. Finally, in  {Section \ref{mains}} we end up proving the main results stated above: Theorems \ref{thm:principal}  and  \ref{thm:principalP}.

\subsection*{Acknowledgments}

 The first author is partially supported by
                      Universidad Nacional de Rosario and 
                      a grant from Consejo Nacional
                      de Investigaciones Cient\'{i}ficas y 
                      T\'{e}cnicas (CONICET).
 The second author is partially supported 
                     by N\'{u}cleo Consolidado de 
                     Matem\'{a}tica Pura y Aplicada, 
                     Universidad Nacional del Centro 
                     de la Provincia de Buenos Aires
                      and by CONICET.
The third author is partially 
                 supported by grants from CONICET
                 and Universidad Nacional del Litoral.

\section{Preliminaries}
\label{prelim}

\subsection{Definition and notations}
\label{prelim:definitions}

\subsubsection{The parabolic setting}
\label{prelim:parabolicsetting}

   The parabolic setting we are considering consists of $\RR^{n+1}$ endowed with the following parabolic metric
\begin{equation*}
d(x,y)=(|x'-y'|^{2}+|t-s|)^{\frac12},
\end{equation*}
 where we write $x=(x',t), y=(y',s)\in\RR^{n+1}$, with $x',y'\in\RR^{n}$ and $t,s\in\RR^+$. We denote the parabolic balls as usual:
\begin{equation*}
B(x,r)=\{ y\in\RR^{n+1}: d(x,y)<r\}.
\end{equation*}
and its Lebesgue measure by $|B(x,r)|= c_n r^{n+2}$.

\subsubsection{The local maximal operator}
\label{prelim:localmaximal}
In this subsection we will denote by $X$ a metric space satisfying the weak homogeneity property, that is, there exists a fix number $N$ such that for any ball $B(x,r)$ there are no more than $N$ points in the ball whose distance from each other is greater than $r/2$. Also $\Lambda$ will mean any open proper and non empty subset of $X$ such that all balls contained in $\Lambda$ are connected sets and $\mu$ will be a Borel measure defined on $\Lambda$ which satisfies a doubling condition on $F_{\beta}$, that is, there is some constant $C_\beta$ such that for any ball $B\in F_{\beta} $
\begin{equation*}
\mu(B)\leq C_\beta \mu(\tfrac{1}{2} B)
\end{equation*}
with $0< \mu(B)<\infty$ for any ball $B\in\mathcal{F}=\bigcup_{0<\alpha<1}\mathcal{F}_\alpha$.

   We shall use the following  local maximal operator associated to $\mathcal{F}_{\beta}$: given $0<\beta<1$ and $\mu$ as above
\begin{equation}\label{maxi}
M_{\mu,\beta}f(x)=\sup_{x\in B\in \mathcal{F}_{\beta}}\frac{1}{\mu(B)}\int_B|f|\; d\mu
\end{equation}
 for any $f\in L^1_{\text{loc}}(\Lambda, d\mu)$ and $x\in \Lambda$. When $\mu$ is the Lebesgue measure we denote $M_{\mu,\beta}f$ by $M_{\beta,\text{loc}}f$.
 
 The boundness property for $M_{\mu,\beta}f$ is contained in the next Theorem:

\begin{thm}[\cite{HSV}, Theorem 1.1]
\label{thm:acot.mxml.loc} Let $X$ and $\Lambda$ as above. Let $0<\beta<1$ and $\mu$ a Borel measure satisfying the doubling property on $\mathcal{F}_{\beta}$.
Then, for $1<p<\infty$, $M_{\mu,\beta}f$ is bounded on $L^{p}_{w}(\Lambda,w d\mu)$ if and only if $w\in A^{\beta}_{p,\text{loc}}(\Lambda)$.
\end{thm}

\subsubsection{The properties of the potential $V$}
\label{prelim:propertiesofV}

   The potential $V$ satisfies assumption (3) and, as it is remarked in \cite{BBHV}, the condition $V\in RH_{q}$ implies that for some $\epsilon >0$ we have also that $V\in RH_{q+\epsilon}$, where the $RH_{q+\epsilon}$ constant of $V$ is controlled in terms of the $RH_{q}$ constant of $V$. They also remark the useful fact that the measure $V(y)dy$ is doubling.

   Associated to the function $V\in RH_{q}$ there is a function $\rho(x)$, called \textit{critical radious}, defined by Shen in \cite{Sh2}:
  \begin{equation}\label{rho}
      \rho(x)=  \sup\bigg\{ r>0: {\dfrac{r^{2}}{|B(x,r)|}}\int _{B(x,r)} V(y)dy\leq 1\bigg\},
  \end{equation}
  which, under our assumptions on $V$, is finite almost everywhere.
  We note that by definition of $\rho$, we have that
\begin{equation}
\label{14} {1\over{\rho(x)^{n-2}}} \int _{B(x,\rho(x))}V(y)dy\leq 1.
\end{equation}

Shen also proved that the following inequalities hold:
\begin{align}
\label{rhoxrhoy} & C\Big( 1+ \frac{|x-y|}{\rho(y)}\Big)^{1\over{k_{0}}}\leq  1+ \frac{|x-y|}{\rho(x)} \leq C \Big( 1+ \frac{|x-y|}{\rho(y)}\Big)^{1\over{k_{0}}},
\end{align}
for some $k_{0}\in\NN$ and any $x,y\in\RR^{n}$ and
  \begin{equation} \label{13}
  {1\over{r^{n}}}\int _{B(x,r)}V(y)dy\leq C \bigg( {R\over r} \bigg)^{n\over q} {1\over{R^{n}}} \int _{B(x,R)} V(y),
  \end{equation}
  for any $0<r<R<\infty$.

\subsubsection{Bounds for the fundamental solutions of the constant coefficient operators $L_{0}$}
\label{prelim:fundamentalsolutions}

   Let us now consider the operator $A$,
   which denotes either $A_{E}$ or $A_{P}$.
   For fixed $x_{0}\in\Lambda$, where $\Lambda$
   denotes $\Omega$ or $\Omega_{T}$, respectively,
    freeze the coefficients $a_{ij}(x_{0})$
    and denote $L_{0}$ the operator $L$ with these constant coefficients.

   Dziubanski in \cite{D} proved that the elliptic operator $L_{0}$ has a fundamental solution $\Gamma(x_{0};x,y)$ which satisfies that for any $k\in\NN$ there exists a constant
    $c_{k}$ (independent of $x_{0}$) such that
\begin{align}
\label{ellipticfundsolbound}
\Gamma(x_{0};x,y) 
   &  \leq  c_{k}{1\over \big({1+{|x-y|\over{\rho(x)}}}\big)^{k}} {1\over{|x-y|^{n-2}}},
\end{align}
for any $x,y\in\RR^{n}$, $x\neq y$. Here $\rho$ is the critical radious associated to $V$ defined in \ref{rho}. We remark that the kernel

\begin{equation*}
W(x,y) = V(y) {1\over {\big( 1+{{|x-y|}\over{\rho(x)}} \big)^{k} }} {1\over{|x-y|^{n-2}}},
\end{equation*}
 satisfies \textit{H\"{o}rmander's condition of order $q$}, briefly \textit{condition $H_{1}(q)$}, in the first variable
(see Proposition 12 in \cite{BBHV}).
This means that there exists a constant $C>0$ such that for any $r>0$ and any $x,x_{0}\in\RR^{n}$ with $|x-x_{0}|<r$, the following inequality holds:
\begin{equation}\label{H1qcondition}
  \sum _{j=1}^{\infty} j |B(x_0, 2^jr)|^\frac{1}{q'}
    \Big( \int _{2^{j}r\leq |x_{0}-y|\leq 2^{j+1}r} |W(x,y)-W(x_{0},y)|^{q} dy \Big)^{1\over q} \leq C.
\end{equation}
Also, observe that from inequalities \ref{rhoxrhoy} we can replace $\rho(y)$ with $\rho(x)$ in the kernel $W$, possibly changing the integer $k$.

   For the parabolic operator $L_{0}$, Kurata showed in Corollary 1 of \cite{K} that it has a fundamental solution $\Gamma(x_{0};x,y)$ which satisfies that for each $k\in\NN$ there exists constants $c_{k}$ and $c_{0}$ (independents of $x_{0}$) such that

\begin{equation*}
\Gamma(x_{0};x,y) \leq  c_{k} {1\over{\big(1+{{d(x,y)}\over{\rho(x')}} \big)^{k}}} {1\over{|t-s|^{n/2}}} e^{-c_{0}{{|x'-y'|^{2}}\over{|t-s|}}},
\end{equation*}
where $d$ is the parabolic distance given in \ref{prelim:parabolicsetting}. Thus,
\begin{align}
\label{parabolicfundsolbound}
\Gamma(x_{0};x,y) & \leq c_{k} {1\over{\big(1+{{d(x,y)}\over{\rho(x')}} \big)^{k}}} {1\over{d(x,y)^{n}}}.
\end{align}
   The parabolic kernel, appearing on the right hand side
   {of \ref{parabolicfundsolbound}}, also satisfies condition $H_{1}(q)$, as we prove in the next subsection.

\subsection{Previous Lemmas}
\label{prelim:lemmas}

\begin{lem}
\label{lemm:parabolickernelh1q}
The kernel
\begin{equation*}
W(x,y)= V(y') {1\over{ \big( 1+ {{d(x,y)}\over{\rho(y')}} \big)^{k}}} {1\over{d(x,y)^{n}}}
\end{equation*}
 satisfies condition $H_{1}(q)$ for $k$ large enough, that is, there exists a constant $C>0$ such that for every $r>0$, $x,x_{0}\in\RR^{n+1}$ with $d(x,x_{0})<r$,
\begin{equation*}
\sum _{j=1}^{\infty} j (2^{j}r)^{{n+2}\over{q'}} \Big( \int _{2^{j}r<d(x_{0},y)\leq 2^{j+1}r} |W(x,y)-W(x_{0},y)|^{q}dy \Big)^{1\over q}\leq C.
\end{equation*}

\end{lem}

\begin{proof}
We follow the lines of Proposition 12 of \cite{BBHV}.
 As usual, we may assume $q>{n\over 2}$. Let $x,x_{0},y\in\Omega_{T}$ be such that $d(x,x_{0})\leq r$ and $d(y,x_{0})\ge 2r$, so that in particular $d(x_{0},y)\simeq d(x,y)$.

   The first step is to compute
\begin{align*}
|W(x,y) -  &  W(x_{0},y)| 
  \leq  V(y') \Bigg( {1\over{ \big( 1 + {{d(x_{0},y)}\over{\rho(y')}}\big)^{k} }} \bigg| {1\over{d(x,y)^{n}}}-{1\over{d(x_{0},y)^{n}}} \bigg| +  \\
&   + {1\over{d(x,y)^{n}}} \bigg| {1\over{ \big( 1 + {{d(x,y)}\over{\rho(y')}}\big)^{k} }}-{1\over{ \big( 1 + {{d(x_{0},y)}\over{\rho(y')}}\big)^{k} }} \bigg| \Bigg) = A+B.
\end{align*}

   We note that by the mean value Theorem

\begin{equation*}
\bigg| {1\over{d(x,y)^{n}}}-{1\over{d(x_{0},y)^{n}}} \bigg| \leq C {{d(x,x_0)}\over{d(x_{0},y)^{n+1}}},
\end{equation*}

   Also

\begin{align*}
\bigg| {1\over{ \big( 1 + {{d(x,y)}\over{\rho(y')}}\big)^{k} }}-{1\over{ \big( 1 + {{d(x_{0},y)}\over{\rho(y')}}\big)^{k} }} \bigg|
 &  \leq C {k\over{\rho(y')}} {{d(x,x_{0})}\over{ \big( 1+{{d(x_{0},y)}\over{\rho(y')}} \big)^{k+1}}}\\
 &  \leq C d(x_{0},y)^{-1} {{d(x,x_{0})}\over{ \big( 1+{{d(x_{0},y)}\over{\rho(y')}} \big)^{k}}},
\end{align*}
 which we obtain from applying again the mean value Theorem.

   Thus, by using the fact that $d(x_{0},y)\simeq d(x,y)$, we obtain that $A$ and $B$ are {bounded} by
\begin{align*}
C V(y') {1\over{ \big( 1 + {{d(x_{0},y)}\over{\rho(y')}}\big)^{k} }}  {{d(x,x_{0})}\over{d(x_{0},y)^{n+1}}}.
\end{align*}

   The second step is to consider the balls
    $B_{j}=B(x_{0},2^{j}r)$,
    the annuli 
    $C_{j}=\{y:2^{j}r<d(y,x_{0})\leq 2^{j+1}r\}=\overline{B_{j+1}}\backslash\overline{B_{j}}$ and the rectangles $B'_{j}\times I_{j}$,
    where
 $B'_{j}=\{y'\in\RR^{n}:|y'- x_0' |\leq 2^{j}r \}$ and $I_{j}=\{s\in\RR:|s-t_0|\leq (2^{j}r)^2\}$. Thus, $C_{j}\subset B'_{j+1}\times I_{j+1}$.\\

  In view of \ref{rhoxrhoy} replacing $\rho(y')$ with $ \rho(x')$ (possibly with a change of the integer $k$), we have that
\begin{align*}
\Big( \int _{C_{j}} A^{q} dy \Big)^{1\over q}
   &  \leq  C {1\over{\big( 1 + {{2^{j}r}\over{\rho(x')}}\big)^{k}}} {r\over{(2^{j}r)^{n+1}}}
           \Big(  \int _{C_{j}} V(y')^{q} dy \Big)^{1\over q}  \\
  & \leq  C {1\over{\big( 1 + {{2^{j}r}\over{\rho(x')}}\big)^{k}}} {r\over{(2^{j}r)^{n+1}}}
      \Big( \int _{I_{j+1}} ds \int _{B'_{j+1}} V^{q}(y') dy' \Big)^{1\over q}   \\
  & \leq  C {1\over{\big( 1 + {{2^{j}r}\over{\rho(x')}}\big)^{k}}} {r\over{(2^{j}r)^{n+1}}} (2^{j+1}r)^{{n+2}\over q} \big( {1\over{|B'_{j+1}|}} \int _{B'_{j+1}} V^{q}(y') dy' \big)^{1\over q}   \\
  &\leq  C {1\over{\big( 1 + {{2^{j}r}\over{\rho(x')}}\big)^{k}}} {r\over{(2^{j}r)^{n+1}}} (2^{j}r)^{{n+2}\over q} {1\over{(2^{j}r)^{n}}} \int _{B'_{j+1}} V(y') dy',
\end{align*}
where in the last inequality we used the reverse H\"{o}lder condition on the potential $V$.

   The third step is to add up and split, as follows:
\begin{align*}
 \sum _{j=0}^{\infty} j (2^{j}r)^{{n+2}\over{q'}}
   & 
    \Big(\int _{C_{j}} A^{q} dy \Big)^{1\over q}\\
   &    \leq   C \sum _{j=0}^{\infty} j (2^{j}r)^{n+2} {1\over{\big( 1 + {{2^{j}r}\over{\rho(x')}} \big)^{k}}} {r\over{(2^{j}r)^{n+1}}} {1\over{(2^{j}r)^{n}}} \int _{B'_{j+1}} V(y')dy' \\
   &  \leq  C \sum _{j:2^{j}r<\rho(x')} \big( \dots \big)
              + C \sum _{j:2^{j}r\ge\rho(x')} \big( \dots \big)  = A_{I}+A_{II}.
\end{align*}

   Therefore,
\begin{align*}
A_{I}
 &  \le
       C \sum _{j:2^{j}r<\rho(x')} j (2^{j}r)^{n+2} {r\over{(2^{j}r)^{n+1}}} {1\over{(2^{j+1}r)^{n}}} \int _{B'_{j+1}} V(y')dy'  \\
 &  \leq   C \sum _{j:2^{j}r<\rho(x')} j (2^{j}r)^{n+2} {r\over{(2^{j}r)^{n+1}}} \Big( {{\rho(x')}\over{2^{j}r}} \Big)^{n\over q} {1\over{\rho(x')^{n}}} \int _{B(x',\rho(x'))} V(y')dy',
\end{align*}
because of equation \ref{13}. Finally, by definition of $\rho$ (see \ref{rho}) and since $q>{n\over 2}$ we conclude that $A_{I}$ is finite:
\begin{align*}
A_{I}  & \leq  C \sum _{j:2^{j}r<\rho(x')} {j\over{2^{j}}}  \Big( {{\rho(x')}\over{2^{j}r}} \Big)^{{n\over q}-2}
\leq  C \sum _{j:2^{j}r<\rho(x')} {j\over{2^{j}}}.
\end{align*}
   Similarly, by using the doubling property of the measure $V(y')dy'$, equation \ref{13} and definition of $\rho$, we have that
\begin{align*}
A_{II}  &\leq C  \sum _{j:2^{j}r\ge\rho(x')} {j\over 2^{j}} (2^{j}r)^{2} \Big({{\rho(x')}\over{2^{j}r}} \Big)^{k}  {1\over{(2^{j}r)^{n}}} \int _{B'_{j+1}} V(y')dy'  \\
 &\leq  C  \sum _{j:2^{j}r\ge\rho(x')} {j\over 2^{j}} (2^{j}r)^{2} \Big({{\rho(x')}\over{2^{j}r}} \Big)^{k}  {1\over{(2^{j}r)^{n}}} \Big({{2^{j}r}\over{\rho(x')}} \Big)^{\alpha} \int _{B(x',\rho(x'))} V(y')dy'   \\
&\leq  C \sum _{j:2^{j}r\ge\rho(x')} {j\over 2^{j}} \Big({{\rho(x')}\over{2^{j}r}} \Big)^{k-\alpha+n-2},
\end{align*}
which is finite for $k$ large enough, and the proof of the Lemma follows.
\end{proof}

\begin{lem}
         \label{lemm:covering.Omega}
				         Let $(X,d)$ be a metric space with the weak homogeneity property (hence separable) and let $\Lambda$ be a nonempty open proper subset of $X$. Let $0 <r_0 < \beta /10$. Then, there exist two families of balls,
         denoted by $\mathcal G_r, \tilde{\mathcal G}_r$,  such that
         \begin{equation*}\mathcal W_{r_0 }=\mathcal G_{r_0 }\cup \tilde{\mathcal G}_{r_0 }=\{B_i\}\end{equation*}
         is a covering of $\Lambda$  by balls of $\mathcal F_\beta $ with the following properties:
         \begin{enumerate}
            \item If $B=B(x_B,s_B)\in \tilde{\mathcal G}_{r_{0}}$,
             then $10B\in \mathcal{F}_\beta $, 
             $d(x_B, \Lambda^{C} )\leq 1 $
              and $\tfrac{1}{2} r_0 d(x_B, \Lambda^{C} )$
              $\leq s_B\leq r_0 d(x_B, \Lambda^{C} )$.
            \item If $B\in\mathcal G_{r_0 }$, then $B \equiv B(x_B,r_0 )$ , $10B\in \mathcal F_\beta $ and $d(x_B, \Lambda^{C} )> 1 $.
            \item If $B,B'\in \mathcal W_{r_0 }$ and $B\cap B'\neq \emptyset $, then: $B\subset 5B'$ and $B'\subset 5B$.
           \item There exists $ M> 0$ such that $\sum _{B\in\mathcal W_{r_0 }} \chi_B(x)\leq M$.
         \end{enumerate}
      \end{lem}

  \begin{proof}
         Let $ r_0<\beta /10$ and define
\begin{equation*}
\Lambda _k =\{x\in \Lambda :  2^{-k}<d(x, \Lambda^{C})\leq 2^{-k+1}\}
\end{equation*}
 for $k>0$, and
\begin{equation*}
 \Lambda _0=\{x\in\Lambda :  1<d(x,\Lambda^{C})  )\}.
\end{equation*}
          We have that          $\Lambda =\bigcup _{i=0}^\infty \Lambda_k$.          For each $k\geq 0$ let us choose a maximal family of points $\{x_{ik}\}_{i=1}^\infty$ in $ \Lambda_k$ such that $ d(x_{ik},x_{ij})> r_0 2^{ -k}.$          For each $k\geq 0$ let us consider the family of balls
$\{B(x_{ik},r_0 2^{-k})\}$.           This family clearly verifies that $ \Lambda_k \subset   \bigcup _{i=1}^\infty B(x_{ik}, r_0 2^{-k})$, and

          \begin{equation*}
           \Lambda =  \bigcup _{k=0}^\infty  \bigcup _{i=1}^\infty B(x_{ik}, r_0 2^{-k}).
          \end{equation*}

          Let us consider for each $k\geq1$ a ball $ B_{ik}= B(x_{ik}, r_{B_{ik}})$ such that $r_{B_{ i_k}}=r_{0} 2^{-k}$. We can easily see that $\{B_{ik}\}$ is a covering of $\Lambda \setminus \Lambda _0$ such that  $10B_{ik}\in \mathcal F_\beta$  {and}
\begin{equation*}
    \dfrac 12 r_0 d(x_{ik},\Lambda^{C} ) < r_{B_{ik}} \leq r_0 d(x_{ik}, \Lambda^{C}).
\end{equation*}

          For  $k=0$ let us consider the family $\{B_{i0}\}=\{B(x_{i0},r_{0})\}_{i=1}^\infty$.     We have that $B_{i0}\in \mathcal F_\beta $ and $10 B_{i0}\in\mathcal F_\beta $.    If $B_{ik}\cap B_{jl}\neq \emptyset $,with  $k,l\geq 0$, then:
          \begin{equation*}B_{jl}\subset 5B_{ik}.\end{equation*} Indeed, if $z\in B_{ik}\cap B_{jl}$, then
\begin{equation*}
2^{-k} \leq  d(x_{ik},\Lambda^{C}) \leq  d(x_{jl},\Lambda^{C}) + d(x_{jl},z) + d(z,x_{ik})\leq 2^{-l+1} + r_{0}2^{-l}+r_{0}2^{k},
\end{equation*}
 from where $2^{-k+l}\leq {{2+r_{0}}\over{1-r_{0}}} < 3$, and by simmetry, also  $2^{-l+k}< 3$, which leads us to $|k-l|\leq 1$. The worst possible situation is $k=l+1$. Let us consider $y\in B_{jl}$, then
\begin{equation*}
d(y,x_{ik})\leq d(y,x_{jl})+d(x_{jl},z)+d(z,x_{ik}) < r_{0}2^{-l}+r_{0}2^{-l}+r_{0}^{-k}=5r_{ik}.
\end{equation*}

            Thus, from the above computations, we can conclude that property 3 holds and $x_{jl}$ is in the same band $\Lambda_{k}$ or in a neighbour band $\Lambda_{j}$. Hence,  the sets $\{x_{jl}\in\Lambda_{j}: B_{ik}\cap B_{jl}\neq \emptyset \}$, with $|k-j|\leq 1$, have at most finite cardinal which does not depend on $B_{ik}$. Then,
there exists $M$, independent of $r_{0}$ and $\beta$, such that
     \begin{equation*}
             \sum_{k=0}^\infty\sum_{i=1}^\infty \chi_{B_{ik}}(x)\leq M.
   \end{equation*}

        \end{proof}

   Let us state the following Lemma, which is often used in the paper without mentioning it.
\begin{lem}
       \label{lemm:tecnico}
        Let $(X,d)$ be a metric space and let $\Lambda$ be a nonempty open proper subset of $X$. Let  $0<\beta<1$ and $\alpha >1$. Given $B_{0}=B(z_{0},r_{0})$ such that  $\alpha B_{0}\in\mathcal{F}_{\beta}$ and any $x\in B_{0}$ we have that $r_{0} < \frac{\beta}{\alpha -\beta} d(x,\Lambda^{C}) $ and $B\big(x,(\alpha - \beta)r_0 \big)\in\mathcal{F}_{\beta}$.
\end{lem}

\begin{proof}
   Since  $\alpha B_{0}\in\mathcal{F}_{\beta}$,
    we have that
\begin{equation*}
  r_{0} <{\frac\beta\alpha} d\big(z_{0},\Lambda^{C}\big)<{\frac\beta\alpha}(d(x,z_{0})+d(x,\Lambda^{C}))<{\frac\beta\alpha}r_{0}+{\frac\beta\alpha}d(x,\Lambda^{C}),
\end{equation*}

 therefore $\big( 1-{\frac\beta\alpha} \big)r_{0}<{\frac\beta\alpha}d(x,\Lambda^{C})$, and finally
 \begin{equation*}
   (\alpha -\beta) r_{0} < \beta d(x,\Lambda^{C}).
\end{equation*}

\end{proof}

We also need the following version of the  Fefferman-Stein  inequality on spaces of homogeneous type:

\begin{lem}[See \cite{PS}] \label{lemm:FS-PS.tipo.homog}  Let $(X,d,\mu)$ be a space of homogeneous type regular in measure, such that $\mu(X)<\infty$. Let $f$ be a positive function in $L^{\infty}$ with bounded support and $w\in A_{\infty}$. Then, for every $p$, $1<p<\infty$, there exists a positive constant $C=C([w]_{A_{\infty}})$ such that if $\|M_{X}f\|_{L^{p}(w)}<+\infty$, then
\begin{equation*}
\|M_{X}f\|_{L^{p}(w)}^{p} \leq  C \|M^{\sharp}_{X}f\|_{L^{p}(w)}^{p},
\end{equation*}
where
\begin{align*}
 M_{X}f(x)
    &=\sup _{x\in P \in {F( X)} } {1\over{\mu(P\cap X|)}}\int _{P \cap X}|f(y)| d\mu(y), \\
 M^{\sharp}_{X}f(x)
    &=\sup _{x\in P \in {F( B)} } {1\over{\mu(P \cap X)}}\int _{P \cap X}|f(y) - f_{P\cap X}| d\mu(y)+ \frac{1}{\mu(X)}\int_{X} f(y) d\mu(y),
\end{align*}
with
\begin{equation*}
 {F( B)}= \{B(x_{B},r_{B}) : x_B \in X, r_{B} >0 \}.
\end{equation*}

\end{lem}

\section{Previous results for the proof of the Theorem \ref{thm:principal}}
\label{previous}

   In order to prove Theorem \ref{thm:principal} we will need the following results:

\begin{thm}[See \cite{CFL1} and \cite{PS}]
      \label{thm:bound.D2u}
          Under assumptions (1) and (2), for any   $p\in(1,\infty)$ and
          $w\in A_{p,\text{loc}}(\Omega )$, there exist  $C$ and $r_0> 0$
            such that for any ball $B_{0}= B(x_0, r_0)$  in  $\Omega$ with $10 B_{0}\in\mathcal{F}_{\beta}$ and any  $u\in W^{2,p}_0(B_0)$ the following inequality holds
           \begin{equation*}\|D^2 u\|_{L_w^p(B_0)}\leq C\|Au\|_{L^p_w(B_0)}.
           \end{equation*}
   \end{thm}
\begin{proof} The proof follows the same lines of the proof of Lemma 4.1 in \cite{CFL1}, which makes use of expansion into spherical harmonics on the unit sphere in $\RR^{n}$.
 After that, all is reduced  to obtain $L^p$- boundedness of\'{o}n-Zygmund operator $T$ and its conmutator on a ball $B$ contained in $\Omega$ (see Theorems 2.10, 2.11 and
 the representation formula (3.1) in this paper). We can look at the operator $T$ and its conmutator $[T,b]$ acting on functions defined over the space of homogeneous type $B$ equipped
 with the Euclidean metric and the restriction of Lebesgue measure. Also, it is easy to check that the weight $w\chi_{B}$ is in $A_p(B)$, provided $w$ belongs to
 $A_{p,\textit{loc}}(\Omega)$, since $B$ has been chosen such that $10B\in\mathcal{F}_{\beta}$. By the weighted theory of singular integrals and conmutators  on  spaces
  of homogeneous type (see for instance \cite{PS}), applied to our operator the result follows.
\end{proof}
     \begin{thm}[See \cite{HSV}, Proposition 4.1]
        \label{thm:bound.epsilon}
              Let $1<p<\infty$ and $w\in A_{p,\text{loc}}(\Omega)$.
             For any function $u\in W^{k,p}_{\delta,w}(\Omega)$,
             and any  $j$, $1\leq j\leq k-1$,
             and $\gamma$ such that $|\gamma|=j$, we have
              \begin{equation*}\|\delta ^jD^\gamma u\|_{L^p_w(\Omega )}\leq C(\epsilon ^{-j}\|u\|_{L^p_w(\Omega )}+\epsilon ^{k-j}\|\delta ^k D^ku\|_{L^p_w(\Omega )}),
            \end{equation*}
     for any $0<\epsilon<1$ and $C$ independent of $u$ and $\epsilon$, with $\delta(x)=\min\{1 , d(x,\Omega)^{C}\}$.
     \end{thm}

     The main Theorem of this section is the following:
 \begin{thm} \label{thm:potencial}
 Let $a_{ij}\in VMO$, for $i,j=1,\dots,n$, $V\in RH_q$ with $1< p\leq q$, and $w\in A_{{{q-1}\over{q-p}}p,\text{loc}}$. Then there exist positive constants $C$ and $r_{0}$
 such that for any ball $B_{0}= B(z_0, r_0)$  in  $\Omega$ with $10 B_{0}\in\mathcal{F}_{\beta}$ and any  $u\in C^{\infty}_{0}(B_0)$, we have that
\begin{equation*}
\|Vu\|_{{L^p_w}(B_0)}  \leq C\|Lu\|_{L^p_w(B_0)}.
\end{equation*}

\end{thm}
\begin{proof}

   For $z_{0}\in\Omega$ pick a ball $B_{0}:=B(z_{0},r_{0})$ with $r_{0}$ to be chosen later. We follow the argument from \cite{BBHV}: let $x_{0}\in B_{0}$ and fix the coefficients of $A$ at $x_{0}$, namely $a_{ij}(x_{0})$, to obtain the operator

\begin{equation*}
L_{0}u=-\sum _{i,j=1}^{n} a_{ij}(x_{0})
   u_{x_{i} x_{j}}  + Vu = A_{0}u+Vu.
\end{equation*}

Rewrite the operator $L_{0}$ in divergence form:

\begin{equation*}
L_{0}u=-\bigg(\sum _{i,j=1}^{n} a_{ij}(x_{0}) u_{x_{i}}\bigg)_{x_{j}} + Vu.
\end{equation*}

   From proposition 4.9 of \cite{D} we know that the operator $L_{0}$ has a fundamental solution $\Gamma(x_{0};x,y)$ which satisfies that for every positive integer $k$ there exists a constant $C_{k}$, independent of $x_{0}$, such that \begin{align}\label{cota.Gamma} \Gamma(x_{0};x,y) & \leq C_{k} {1\over {\big( 1+{{|x-y|}\over{\rho(x)}} \big)^{k} }} {1\over{|x-y|^{n-2}}},\end{align} where $\rho(x)$ is the critical radius (recall section \ref{prelim:definitions}).

   Thus, for any $u\in C^{\infty}_{0}(B_{0})$, $x\in B_{0}$,
\begin{align*}
u(x) = & \int  \Gamma(x_{0};x,y) L_{0}u(y)dy= \\
 = & \int  \Gamma(x_{0};x,y)Lu(y)dy + \int  \Gamma(x_{0};x,y) [A_{0}u(y)-Au(y)] dy.\\
\end{align*}

   Now if we let $x_{0}=x$, we obtain
\begin{align}
\label{u.defreeze}
u(x) = & \int  \Gamma(x;x,y) Lu(y) dy + \sum _{i,j=1}^{n} \int  \Gamma(x;x,y) [a_{ij}(y)-a_{ij}(x)] u_{x_{i}x_{j}}(y) dy.
\end{align}
Then the following pointwise bound holds for all $k\in\NN$, $x\in B_{0}$,
\begin{align}
\label{Vu}
|V(x)u(x)|
   &  \leq C_{k} V(x) \int  _{B_{0}} {1\over {\big( 1+{{|x-y|}\over{\rho(x)}} \big)^{k} }} {1\over{|x-y|^{n-2}}}
   \Big( |Lu(y)| + \\
\notag
       & \qquad \qquad \qquad \qquad \qquad
        + \sum _{i,j=1}^{n} |a_{ij}(y)-a_{ij}(x)| 
|u_{x_{i}x_{j}}(y)| \Big) dy.
\end{align}

   Next let us rewrite (\ref{Vu}) as
\begin{align}
\label{9}
|V(x)u(x)| \leq C_{k} S_{k}(|Lu|)(x)+\sum _{i,j=1}^{n} S_{k,a_{ij}}(|u_{x_{i}x_{j}}|)(x),
\end{align}
where $S_{k}$ and $S_{k,a}$ are the integral operators defined as
\begin{align}
\label{Sk}
S_{k}f(x) = & V(x) \int  {1\over {\big( 1+{{|x-y|}\over{\rho(x)}} \big)^{k} }} {1\over{|x-y|^{n-2}}} f(y) dy,
\end{align}
and
\begin{align}
\label{Ska}
S_{k,a}f(x) = & V(x) \int  {1\over {\big( 1+{{|x-y|}\over{\rho(x)}} \big)^{k} }} {1\over{|x-y|^{n-2}}} |a(y)-a(x)| f(y) dy,
\end{align}
with $a\in L^{\infty}\cap VMO(\RR^{n})$, $k\in\NN$.

   We will prove in Theorem \ref{thm:Sk.sin.conmutador} below that for all $p\in (1,q]$ and $k$ large enough,
\begin{align}
\label{10}
\|S_{k}f\|_{L^{p}_{w}(B_{0})} \leq C \|f\|_{L^{p}_{w}(B_{0})}.
\end{align}
Also, we will prove in Theorem \ref{thm:Ska.con.conmutador} below that for each $\epsilon >0$ there exists $r_{0}>0$ depending on the VMO-modulus of the function $a$ such that
\begin{align}
\label{11}
\|S_{k,a}f\|_{L^{p}_{w}(B_{0})} \leq \epsilon \|f\|_{L^{p}_{w}(B_{0})}.
\end{align}

   Then, by (\ref{9}), (\ref{10}), (\ref{11}) and Theorem \ref{thm:bound.D2u} we have that for any $u\in C_{0}^{\infty}(B_{0})$ with $r_{0}$ small enough,
\begin{align*}
\|Vu\|_{L^{p}_{w}(B_{0})} & \leq C \|Lu\|_{L^{p}_{w}(B_{0})} + \epsilon \|u_{x_{i}x_{j}}\|_{L^{p}_{w}(B_{0})} \leq C \|Lu\|_{L^{p}_{w}(B_{0})} + C \epsilon \|Au\|_{L^{p}_{w}(B_{0})}  \\
 & \leq (C+C\epsilon) \|Lu\|_{L^{p}_{w}(B_{0})} + C \epsilon \|Vu\|_{L^{p}_{w}(B_{0})},
\end{align*}
and Theorem \ref{thm:potencial} follows.

\end{proof}

\subsection{Statement and proof of Theorems \ref{thm:Sk.sin.conmutador} and \ref{thm:Ska.con.conmutador}:}

   Following the lines of \cite{BBHV}, let us also consider the operators defined in $\Omega $

 \begin{align*}
S^{\ast}_{k}f(x)
   &=\int  {{V(y)}\over{\big( 1+ {{|x-y|}\over{\rho(y)}}\big)^{k}}} {1\over{|x-y|^{n-2}}} f(y)dy, \qquad x\in\Omega.  \qquad \mbox{ and}\\
S^{\ast}_{k,a}f(x) &=  \int {{V(y)}\over{\big( 1+ {{|x-y|}\over{\rho(y)}}\big)^{k}}} {1\over{|x-y|^{n-2}}} |a(y)-a(x)| f(y)dy,
\end{align*}
for each positive integer $k$ and $a\in VMO$.
These operators are the adjoint of the integral operator $S_{k}$ and $S_{k,a}$, given in  \eqref{Sk} and \eqref{Ska} respectively.

\begin{thm}
    \label{thm:Sk.sin.conmutador}
   Let $B_{0}$ be a ball in $\mathcal{F}_{\beta}$
  such that $10 B_{0}\in\mathcal{F}_{\beta}$.
  Then for $k$ large enough and $p \in [1,q]$,
  the operator $S_{k}$ is bounded on $L^{p}_{w}(B_{0})$, with $w\in A_{{{q-1}\over{q-p}}p ,\text{loc}}(\Omega)$.
  \end{thm}

 \begin{proof}
    It is enough to prove that the adjoint operator $S^{\ast}_{k}$
 is bounded on $L^{p'}_{v}(B_{r_{0}})$, with $v = w^{-1/p-1}\in A_{p'/q',\text{loc}}(\Omega)$ for $p' \in [q',\infty]$, since $p'\over q'$ and ${{q-1}\over{q-p}}p$ are conjugate exponents.
 As we pointed out in section \ref{prelim:fundamentalsolutions}, we may replace $\rho(y)$ by $\rho(x)$ in the kernel of the operator $S^{\ast}_{k}$ (and maybe changing the integer $k$).
 Assume, without loss of generality, that $f\ge 0$.
 Also assume that $q>{n\over 2}$, which  can be done because of the fact that if $V$ satisfies the $RH_{q}$ property, then $V$ satisfies the $RH_{q+\epsilon}$ property for some $\epsilon >0$.

   We will prove the pointwise bound
\begin{equation*}
S^{\ast}_{k}f(x)\leq C (M_{\beta,\text{loc}}(|f|^{q'})(x))^{{1\over{q'}}} =:M_{q',\text{loc}},
\end{equation*}
  for $x\in B_{0}$, $f\in L^{p}_{w}(B_{0})$ and $f\ge 0$. If $p>q'$ the theorem then follows by the boundedness of the local-maximal function (Theorem \ref{thm:acot.mxml.loc}),
   and if $p=q'$ the theorem follows from the fact that $V$ satisfies the $RH_{q+\epsilon}$ property for some $\epsilon >0$.

   We have that
\begin{align*}
S^{\ast}_{k}f(x)
     & \leq  C \int _{|x-y|<\rho(x)}{{V(y)}\over{\big( 1+ {{|x-y|}\over{\rho(x)}}\big)^{k}}} {1\over{|x-y|^{n-2}}} \chi_{B_{0}}(y)f(y) dy \, + \\
& \qquad
   + C \int _{|x-y|\ge\rho(x)}{{V(y)}\over{\big( 1+ {{|x-y|}\over{\rho(x)}}\big)^{k}}} {1\over{|x-y|^{n-2}}} \chi_{B_{0}}(y)f(y) dy   \\
& \leq  C \int _{|x-y|<\rho(x)}{{V(y)}\over{|x-y|^{n-2}}} \chi_{B_{0}}(y)f(y) dy \, + \\
&  \qquad 
  + C \int _{|x-y|\ge\rho(x)} \Big({{\rho(x)}\over{|x-y|}} \Big)^{k} {{V(y)}\over{|x-y|^{n-2}}} \chi_{B_{0}}(y)f(y) dy = \mathbf{A}(x)+\mathbf{B}(x).\\
\end{align*}

 Let $x\in B_{0}=B(z_{0},r_{0})$.

   Let us first study $\mathbf{A}(x)$. Denote by $B_{j}$ the balls $B_{j}=B(x,2^{-j}\rho(x))$ and by $C_{j}$ the annuli defined as $C_{j}=\{ y: 2^{-(j+1)}\rho(x)<|x-y|\leq 2^{-j}\rho(x) \}=\overline{B_{j}}\backslash \overline{B_{j+1}}$, $j\in\NN_{0}$.

   If $\rho(x)\leq r_{0}$ then, by the Lemma \ref{lemm:tecnico} we have that $\rho(x)\leq  r_{0}<{\beta\over{10-\beta}}d(x,\Omega^{C})$. Then $B(x,\rho(x))\in\mathcal{F}_{\beta}$ and we proceed as in \cite{BBHV}. That is,
\begin{align*}
\mathbf{A}(x)  & \leq  C \sum _{j=0}^{\infty} {1\over{(2^{-j}\rho(x))^{n-2}}}
               \int _{C_{j}} V(y)f(y) dy \leq \\
& \leq C \sum _{j=0}^{\infty} (2^{-j}\rho(x))^{2}
    \bigg( {1\over{|B_{j}|}} \int _{B_{j}} V(y)^{q} dy  \bigg)^{1\over q}
    \bigg(  {1\over{|B_{j}|}} \int _{B_{j}} f(y)^{q'}dy \bigg)^{1\over q'}   \\
&\leq  C M_{q',\text{loc}}(f)(x)\sum _{j=0}^{\infty} (2^{-j}\rho(x))^{2}
   \bigg( {1\over{|B_{j}|}} \int _{B_{j}} V(y) dy  \bigg),
\end{align*}
by H\"{o}lder inequality, $RH_{q}$ condition and the definition of local Maximal function of exponent $q'$.

   A slight modification of the argument is needed in the case $\rho(x) > r_{0}$: there exists $j_{0}\in\NN_{0}$ such that $2^{-(j_{0}+1)}\rho(x)< r_{0} \leq 2^{-j_{0}}\rho(x)$. Let $y\in C_{j}$, for $j\leq j_{0}-2$. Then,
\begin{equation*}
2^{-(j+1)}\rho(x)<|x-y|\leq 2^{-j}\rho(x),
\end{equation*}
 and also
\begin{equation*}
2r_{0}< 2^{-j_{0}+1}\rho(x)\leq 2^{-(j+1)}\rho(x),
\end{equation*}
 from where
\begin{equation*}
2r_{0}\leq 2^{-(j+1)}\rho(x)<|x-y|\leq |x-z_{0}|+|z_{0}-y|<r_{0} + |z_{0}-y|.
\end{equation*}
 Therefore $|z_{0}-y|>r_{0}$,
 and thus $B_{0}\cap C_{j}=\emptyset$ if $j\leq j_{0}-2$.
 Then, 
\begin{align*}
\mathbf{A}(x)
&  \leq C \sum _{j=j_{0}-1}^{\infty} {1\over{(2^{-j}\rho(x))^{n-2}}} \int _{B_{0}\cap C_{j}} V(y)f(y) \, dy  \\
&\leq  C \sum _{j=j_{0}-1}^{\infty} (2^{-j}\rho(x))^{2} \big( {1\over{|B_{j}|}} \int _{B_{j}} V(y)^{q} dy  \big)^{1\over q} \bigg(  {1\over{|B_{j}|}} \int _{B_{j}} f(y)^{q'} \, dy \bigg)^{1\over q'},
\end{align*}
by H\"{o}lder inequality and the fact that
$C_{j}\subset \overline{B_{j}}$.
Since $B_{j}=B(x,2^{-j}\rho(x))\subset B(x,4r_{0}) \subset B(z_{0},5r_{0})$ and $B(z_{0},10r_{0})\in\mathcal{F}_{\beta}$, we have that $B_{j}\in\mathcal{F}_{\beta}$, $j\ge j_{0}-1$,
in view of Lemma \ref{lemm:tecnico}.

   Then, applying the $RH_{q}$ condition on $V$, we obtain
\begin{equation*}
\mathbf{A}(x)\leq C M_{q',\text{loc}}(f)(x)\sum _{j=j_{0}-1}^{\infty} (2^{-j}\rho(x))^{2}\bigg( {1\over{|B_{j}|}} \int _{B_{j}} V(y) dy  \bigg).
\end{equation*}
 Finally, we follow the same steps as in \cite{BBHV} to conclude that
\begin{equation*}
\mathbf{A}(x)\leq C M_{q',\text{loc}}(f)(x),
\end{equation*}
namely, choose $R=\rho(x)$ and $r=2^{-j}\rho(x)$ in \ref{13}, and use \ref{14} from section \ref{prelim:propertiesofV}, when needed.

   Next we study $\mathbf{B}(x)$.

   This time, if $\rho(x) >2r_{0}$ we have that $\mathbf{B}(x)=0$.

   The other case goes as follows: now consider the balls 
   $B_{j}=B(x,2^{j}\rho(x))$ 
   and the annuli 
   $C_{j}=\{y:2^{j-1}\rho(x)<|x-y|\leq 2^{j}\rho(x)\} \subset \overline{B_{j}}\backslash\overline{B_{j-1}}$,
    for $j\in\NN_{0}$. 
    There exists 
    $j_{0}\in\NN_{0}$ such that $2^{j_{0}-1}\rho(x)<r_{0}\leq 2^{j_{0}}\rho(x)$. 
    Consider $y\in C_{j}$ for $j\ge j_{0}+2$. Then,
\begin{equation*}
2^{j-1}\rho(x)<|x-y|\leq 2^{j}\rho(x),
\end{equation*}
 and since $2r_{0}\leq 2^{j_{0}+1}\rho(x)\leq 2^{j-1}\rho(x)$, we have that
\begin{equation*}
2r_{0}<|x-y|\leq |x-z_{0}|+|z_{0}-y|<r_{0}+|z_{0}-y|.
\end{equation*}
 Therefore, $|z_{0}-y|>r_{0}$ and we conclude that $B_{0}\cap C_{j}=\emptyset$, for $j\ge j_{0}+2$. Then, 
\begin{align*}
\mathbf{B}(x) &
  \leq   C \sum _{j=0}^{j_{0}+1} {{2^{-jk}}\over{(2^{j}\rho(x))^{n-2}}} \int _{B_{0}\cap C_{j}} V(y)f(y) dy \\
 & \leq  C \sum _{j=0}^{j_{0}+1} {{(2^{j}\rho(x))^{2}}\over{2^{jk}}} \Big( {1\over{|B_{j}|}} \int _{B_{j}} V(y)^{q}dy\Big)^{1\over q} \Big( {1\over{|B_{j}|}} \int _{B_{j}} f(y)^{q'}dy\Big)^{1\over q'},
\end{align*}
by H\"older inequality and the fact that $C_{j}\subset B_{j}$.
Then, for $0\leq j\leq j_{0}+1$,
we have that $B(x,2^{j}\rho(x))\subset B(x,4r_{0}) \subset B(z_{0},5r_{0})$.
Again, since $ B(z_{0},10r_{0})\in\mathcal{F}_{\beta}$, we get $B_{j}\in\mathcal{F}_{\beta}$. Thus, from the $RH_{q}$ condition
\begin{equation*}
\mathbf{B}(x)\leq C M_{q',\text{loc}}(f)(x) \sum _{j=0}^{j_{0}+1} {{(2^{j}\rho(x))^{2}}\over{2^{jk}}}  \Big( {1\over{|B_{j}|}} \int _{B_{j}} V(y)dy\Big).
\end{equation*}

 Now we continue the proof given in \cite{BBHV}, that is, use again \ref{13} and \ref{14}, to conclude that

\begin{equation*}
\mathbf{B}(x)\leq C M_{q',\text{loc}}(f)(x).
\end{equation*}

\end{proof}

\begin{thm}\label{thm:Ska.con.conmutador}
   Let $p \in (1,q]$ and $w\in A_{{{q-1}\over{q-p}}p,\text{loc}}(\Omega)$. Then, given $\epsilon >0$ there exist $r_{0}>0$, depending on the $VMO-$modulus of $a$, such that for any
 ball $B_{0}= B(z_0, r_0)$  in  $\Omega$ with $10 B_{0}\in\mathcal{F}_{\beta}$, the inequality
\begin{align*}
  {\|S_{k,a}f\|_{L^{p}_{w}(B_{0})} \leq \epsilon \|f\|_{L^{p}_{w}(B_	{0})} }
\end{align*}
holds for all $f\in L^{p}_{w}(B_{0})$ and $k$ large enough.
\end{thm}

   Now we can write
\begin{equation*}
S^{\ast}_{k,a}f(x) = \int |a(y)-a(x)|W(x,y)f(y)dy,
\end{equation*}
where $W(x,y)$ is the kernel given in Lemma \ref{lemm:parabolickernelh1q}  which satisfies the $H_{1}(q)$ condition,
      and we deduce  Theorem \ref{thm:Ska.con.conmutador}, from the following abstract result:

   \begin{thm}
  \label{thm:abstracto}
  Let $w\in A_{p/q',\text{loc}}(\Lambda)$ with $q'<p<\infty$ and $\Lambda= \Omega$ or $\Omega_T$. Let $B_0$ be a ball in $\Lambda$ such that $10 B_{0}\in\mathcal{F}_{\beta}$. Assume that
   $W(x,y)$ is a non-negative kernel satisfying the $H_{1}(q)$ condition on the first variable, for some $q>1$ such that the operator
\begin{equation*}
Tf(x)=\int W(x,y) f(y) dy
\end{equation*}
 is bounded on $L^{p}_{w}(B_{0})$. Then for $b\in BMO(\RR^n) $ or $BMO(\RR^{n+1})$ the operator``positive commutator''
\begin{equation*}
T_{b}f(x)=\int _{B_{0}} |b(x)-b(y)| W(x,y)f(y)dy
\end{equation*}
 is bounded on $L^{p}_{w}(B_{0})$, and
\begin{equation*}
\|T_{b}f\|_{L^{p}_{w}(B_{0})}\leq C \|b\|_{BMO} \|f\|_{L^{p}_{w}(B_{0})}.
\end{equation*}

\end{thm}

\begin{proof}

    In view of Lemma \ref{lemm:FS-PS.tipo.homog}, we will prove the following pointwise inequality: for $s>q'$ there exists a constant $C>0$ independent of $b$ and $f$ such that
\begin{align}\label{18}
\ M_{B_0}^{\sharp} (T_{b}f )(x)\leq C \|b\|_{BMO} [ M_{s,\text{loc}}(Tf)(x)+M_{s,\text{loc}}(f)(x)],
\end{align}
for all $x\in B_{0}$, where


\begin{equation*}
M^{\sharp}_{B_0}f(x)= \sup _{ x\in B,  x_B \in B_0} \inf _{c>0} {1\over{|B\cap B_0|}} \int _{B\cap B_0} |f(y)-c| dy  + {1\over{|B_{0}|}}\int _{B_{0}}|f(y)|dy.
\end{equation*}

Fixed $x\in B_0$ and choose $B=B(x_{B },r_B)$ with $x\in B$ and $x_B\in B_0$. Thus $ |B| \simeq |B\cap B_{0}|$. Let $\widetilde{B}=2B=B(x_{B},2r_{B})$. From Lemma \ref{lemm:tecnico} it follows that
$\widetilde{B}\in \mathcal{F}_{\beta}$. Now for a positive function $f$ let us split it into the sum $f=f_{1}+f_{2}$,
where $f_{1}=f\chi_{\widetilde{B}}$ and $f_{2}=f\chi_{\widetilde{B}^{C}}$.

   Proceeding as in \cite{BBHV}, we obtain the expression
\begin{align*}
|T_{b}f(y) - C_{B}| 
   & \leq |b(y)-b_{B}|  \,
        Tf(y) + T(|b-b_{B}|f_{1})(y) \\
    & \qquad\qquad\qquad
           + \int _{B_{0}} |W(y,z)-W(x_{B},z)| 
    |b(z)-b_{B}| f_{2}(z) dz  \\
 &= \mathbf{A}(y) + \mathbf{B}(y) + \mathbf{C}(y)
\end{align*}
for any $y\in B$, where $c_{B}=T(|b-b_{B}|f_{2})(x_{B})=\int _{B_{0}} |b(z)-b_{B}| W(x_{B},z) f_{2}(y)dz$.

   Let us first bound $\mathbf{A}(y)$. Taking average over $B\cap B_0$, for $s>q'$,
\begin{align*}
Av(\mathbf{A})
   &  =  {1\over{|B\cap B_0|}} \int _{B\cap B_0} |b(y)-b_{B}|Tf(y)dy\\
      & \leq C\Big( {1\over{|B|}} \int _{B} |b(y)-b_{B}|^{s'}dy \Big)^{1\over s'}
           \Big( {1\over{|B|}} \int _{B} \chi_ {B_0}|Tf(y)|^{s} dy \Big)^{1\over s} \leq \\
& \leq  C\|b\|_{BMO} M_{s,\text{loc}}(\chi_ {B_0} Tf)(x),
\end{align*}

   Choose now $\gamma$ such that $s>\gamma>q'$. The computations for the average of $\mathbf{B}$ from \cite{BBHV} also hold in our case:
\begin{align*}
Av(\mathbf{B})
   & \leq  {C\over{|B|}}\int _{B}\chi_ {B_0}T(|b-b_{B}|f_{1})(x)dx
      \leq C\Big( {1\over{|B|}}\int _{B}T(|b-b_{B}|f_{1})^{\gamma}(x)dx \Big)^{{1\over{\gamma}}} \leq \\
& \leq C \Big( {1\over{|B|}}\int _{\widetilde{B}} |b(x)-b_{B}|^{\gamma} |f_{1}(x)|^{\gamma} dx \Big)^{{1\over{\gamma}}},
\end{align*}
since $T$ is bounded on $L^{p}(\RR^n)$(see Theorem 3.1 in \cite{Sh2} and Theorem 5 in \cite{BBHV}).
Then, by H\"{o}lder's inequality,
\begin{align*}
Av(\mathbf{B}) & \leq C \bigg( {1\over{|\widetilde{B}|}} \int _{\widetilde{B}} |f(x)|^{s} dx \bigg)^{1\over s}
   \bigg( {1\over{|B|}} \int _{\widetilde{B}} |b(x)-b_{B}|^{\gamma  ( {s\over{\gamma}} )'} \bigg)^{1\over{\gamma ( {s\over{\gamma}} )'}}  \leq \\
   & \leq C \bigg( {1\over{|\widetilde{B}|}} \int _{\widetilde{B}} |f(x)|^{s} dx \bigg)^{1\over s} \bigg[ \bigg( {1\over{|\widetilde{B}|}} \int _{\widetilde{B}} |b(x)-b_{\widetilde{B}}|^{\gamma ( {s\over{\gamma}} )'} \bigg)^{1\over{\gamma ( {s\over{\gamma}} )'}} + |b_{B}-b_{\widetilde{B}}|  \bigg]  \leq \\
& \leq C \|b\|_{BMO} M_{s,\text{loc}}(f)(x),
\end{align*}
because $|b_{B}-b_{\widetilde{B}}|\leq C \|b\|_{BMO}$ and the John-Nirenberg inequality.

   Next we choose $\gamma$ such that ${1\over\gamma}+{1\over q} +{1\over s} =1$, and we define the balls $B_{j}=B(x_{B},2^{j}r)$ and the annuli $C_{j}=\{z:2^{j-1}r<|x_{B}-z|\leq 2^{j}r\}$. Like in the proof of theorem \ref{thm:Sk.sin.conmutador}, there exists $j_{0}\in\NN_{0}$ such that $C_{j_{0}}\cap B_{0} \neq \emptyset$ and $C_{j_{0}+1}\cap B_{0} = \emptyset$, then by Lemma \ref{lemm:tecnico}, we have that $B_{j}\in\mathcal{F}_{\beta}$ for $j\leq j_{0}$.  Then, for any $y\in B$, we have that
\begin{align*}
\mathbf{C}(y)
  & = \int _{\widetilde{B}^{C}\cap B_{0}} 
  |b(z)-b_{B}| 
  |W(y,z)-W(x_{B},z)| f(z) dz \\
 &  \leq   \sum _{j=2}^{j_{0}} 
   \int _{C_{j}\cap B_{0}} |b(z)-b_{B}|
    |W(y,z)-W(x_{B},z)| f(z)dz \leq \\
& \leq  C \sum _{j=2}^{j_{0}}
 \bigg( {1\over{|B_{j}|}} \int _{B_{j}} |b(z)-b_{B}|^{\gamma} dz \bigg)^{1\over{\gamma}}
\\
&\qquad\qquad\quad
  \bigg( {1\over{|B_{j}|}} \int _{C_{j}} |W(y,z)-W (x_{B},z)|^{q}d \bigg)^{1\over q} \bigg( {1\over{|B_{j}|}} \int _{B_{j}} |\chi_{B_{0}}f(z)|^{s} dz \bigg)^{1\over s}  \leq \\
 &\leq  C \sum _{j=2}^{j_{0}} |B_j| \bigg[ \bigg( {1\over{|B_{j}|}} \int _{B_{j}} |b(z)-b_{B_{j}}|^{\gamma} dz \bigg)^{1\over{\gamma}} + |b_{B}-b_{B_{j}}| \bigg]
 \\
 &\qquad\qquad\quad
 \bigg( {1\over{|B_{j}|}} \int _{C_{j}} |W(y,z)-W(x_{B},z)|^{q}dz \bigg)^{1\over q} M_{s,\text{loc}}(f)(x) \leq \\
& \leq  C \|b\|_{BMO} M_{s,\text{loc}}(f)(x) \sum _{j=2}^{\infty} (2^{j}r)^{n\over{q'}}j \bigg( \int _{C_{j}} |W(y,z)-W(x_{B},z)|^{q}dz \bigg)^{1\over q} \leq \\
&\leq  C \|b\|_{BMO} M_{s,\text{loc}}(f)(x),
\end{align*}
because of the $H_{1}(q)$ condition, the John-Nirenberg inequality
 and the fact that $|b_{B}-b_{B_{j}}|\leq Cj\|b\|_{BMO}$.
 Then putting together all the above estimates, we get
\begin{equation*}
 \sup _{\substack{x\in B  \\  x_B \in B_0}}
    \inf _{c>0} 
   {1\over{|B\cap B_0|}}
      \int _{B\cap B_0} |f(y)-c| \,  dy 
      \leq C \|b\|_{BMO} \Big(M_{s,\text{loc}}(f)(x) + M_{s,\text{loc}}(Tf)(x)\Big).
\end{equation*}
On the other hand, proceeding as above we also have
\begin{align*}
{1\over{|B_{0}|}}\int _{B_{0}}|T_{b}f(y)|dy
 &\leq  {1\over{|B_{0}|}}
   \int _{B_{0}}(|b(y)-b_{B_0}|Tf(y) + T(|b-b_{B_0}|f)(y)) \, dy\\
  & \leq  C \|b\|_{BMO} \Big(M_{s,\text{loc}}(\chi_{B_0} Tf)(x) + M_{s,\text{loc}}(f)(x)\Big)
\end{align*}

   Thus we obtain \ref{18}, which together with Lemma \ref{lemm:FS-PS.tipo.homog} and Theorem \ref{thm:acot.mxml.loc} imply the Theorem.

\end{proof}

\begin{proof}[Proof of Theorem \ref{thm:Ska.con.conmutador}]
  By duality, we prove the theorem for the adjoint operator $S^{\ast}_{k,a}$ with $v = w^{-1/p-1}\in A_{p'/q',\text{loc}}(\Omega)$ for $p' \in [q',\infty)$.

Applying Theorem \ref{thm:abstracto} to the operator $S^{\ast}_{k,a}$ for $k$ large enough we get that if $q'<p'<\infty$,
\begin{equation*}
\|S^{\ast}_{k,a}f\|_{L^{p'}_{v}(B_{0})} \leq C \|a\|_{BMO} \|f\|_{L^{p'}_{v}(B_{0})},
\end{equation*}
 and if $p'=q'$ we use again that $V\in RH_{q+\epsilon}$.

   Since $a\in VMO(\RR^{n} )$, there exists a bounded uniformly continuous function $\phi$ in $\RR^{n}$ such that \\
$\|a-\phi\|_{BMO}<\epsilon$. Also, for $z_{0}\in \Omega$ and $r_0 >0$ there exists a uniformly continuous function $\psi$ such that $ \psi = \phi$ in $B _0 =B(z_0,r_0)$ and
\begin{equation*}
\|\psi\|_{BMO}\leq \omega_\phi (2r_0),
\end{equation*}
where $\omega_\phi (2r_0)$ denote the modulus of continuity of $\phi$ (see \cite{CFL1}). Choosing $r_0$ small enough, for all $f \in L^{p}_{v}(B_{0})$, we have
\begin{align*}
\|S^{\ast}_{k,a}f\|_{L^{p'}_{v}(B_{0})}
   &\leq  \|S^{\ast}_{k,a-\phi}f\|_{L^{p'}_{v}(B_{0})} + \|S^{\ast}_{k,\phi}f\|_{L^{p'}_{v}(B_{0})}
   \\
   &=
   \|S^{\ast}_{k,a-\phi}f\|_{L^{p'}_{v}(B_{0})} + \|S^{\ast}_{k,\psi}f\|_{L^{p'}_{v}(B_{0})}\\
& \leq C \|a-\phi\|_{BMO} \|f\|_{L^{p'}_{v}(B_{0})} + C \|\psi\|_{BMO} \|f\|_{L^{p'}_{v}(B_{0})} \\
 & \leq C \epsilon \|f\|_{L^{p'}_{v}(B_{0})},
\end{align*}
thus, the Theorem follows.

\end{proof}

\section{ Previous results for the proof of the Theorem \ref{thm:principalP}}
\label{previousP}

   We now present the parabolic-interpolation Theorem, which makes use of the Theorem \ref{thm:acot.mxml.loc}.

\begin{thm}
        \label{thm:bound.epsilonP}
   Let $1<p<\infty$ 
   and $w\in A_{p,\text{loc}}(\Omega_T)$.
 For any function $u\in W^{k,p}_{\delta,w}(\Omega_T)$, 
 any  $j$, $1\leq j\leq k-1$, 
 and $\gamma$ such that $|\gamma|=j$, 
 we have that
 \begin{equation}\label{interin}
 \|\delta ^jD^\gamma u\|_{L^p_w(\Omega_T )}\leq C(\epsilon ^{-j}\|u\|_{L^p_w(\Omega_T )}+\epsilon ^{k-j}\|\delta ^k D^ku\|_{L^p_w(\Omega_T )}).
 \end{equation}
for any $0<\epsilon<1$ and $C$ independent of $u$ and $\epsilon$ with $\delta (x',t)=\min\{1,d((x',t),\Omega_{T}^C )\}$, where $D^\gamma$ denotes the derivative with
respect to the first variable.
 \end{thm}
\begin{proof}
 The proof follows the same lines of the proof of Theorem \ref{thm:bound.epsilon} of \cite{HSV} with appropriate changes. We include it for completeness.
  We  consider the following Sobolev's integral representation (see \cite{B}):
\begin{equation*}
|D^\gamma v(x',s)|\leq C\bigg(\sigma^{-n-j}\int_{B(x',\sigma)}|v(y',s)|\,+\int_{B(x',\sigma)}\frac{|D^k v(y',s)|}{|x'-y'|^{n-k+j}}dy'\bigg),
\end{equation*}
for any $\sigma>0$,  $(x',s) \in\mathbb{R}^{n}\times(0,T)$ and $v\in W_{\text{loc}}^{k,1}(\mathbb{R}^{(n+1)})$.

   Let us choose a Whitney' type covering $\mathcal W_{r_0 }$ of $\Omega_T$ with $\beta=1/2$ and $r_0<1/20$. For $P=B(x_P,r_P)\in\mathcal W_{r_0 }$,
    take a $\mathcal{C}_0^\infty$ function $\eta_P$ such that $\hbox{supp}(\eta_P)\subset 4P\subset\Omega_T$, $0\leq\eta_P\leq 1$, and $\eta_P\equiv 1$ on $2P$.

   We apply now the above inequality to $u\eta_P$ which, by our assumptions, belongs to  $W_{\text{loc}}^{k,1}(\mathbb{R}^n)$.
   Observe that for $(x',s)\in P$ and $\sigma\leq r_P$ we have $B((x',s),\sigma)\subset 2P$ and consequently $u\eta_P$
   as well as its derivatives coincide with $u$ and its derivatives when integrated over such balls.

   Therefore for $(x',s)\in P$ and $\sigma\leq r_P$, we obtain the above inequality with $v$ replaced by $u$, namely
\begin{align}\label{pwinter}
| D^\gamma u(x',s)|
  &=| D^\gamma (u\eta_P)(x',s) |\\
  \notag
     &\leq C \sigma^{-n-j}\int_{B(x',\sigma)}|u(y',s)|dy'\,+C\int_{B(x',\sigma)}\frac{|D^k u(y',s)|}{|x'-y'|^{n-k+j}}dy' .
\end{align}

   Moreover, as is easy to check from the properties of the covering $\mathcal W_{r_0 }$, the balls $B(x,\sqrt{2}\sigma)$, for $x\in P$ and $\sqrt{2}\sigma\leq r_P$,
   belong to the family $\efe$ for $\beta=1/2$. In fact, for $x\in P$, since from properties 1 and 2 of Whitney's Lemma  we get $10 P\in \efe$, applying the Lemma \ref{lemm:tecnico} we get
\begin{equation*}
B(x,\sqrt{2}\sigma)\subset B(x,(10-\beta)\sqrt{2}\sigma)\subset B(x,(10-\beta)r_P)\in\efe.
\end{equation*}
 Let $x=(x',t)\in P.$ Integrating in \eqref{pwinter} over $I_\sigma(t)= (t-\sigma^2,t+\sigma^2)$ and noticing that $B(x',\sigma)\times I_\sigma(t)\subset B(x,\sqrt{2}\sigma)\in \efe, $
we get
\begin{align*}
\sigma^{-2}
  \!\! 
  \int_{I_\sigma(t)}
  &
   | D^\gamma u(x',s) | \,  ds \\
   &\leq C \sigma^{-n-2-j}   
      \!\!\!\!\!
     \iint \limits _{B(x',\sigma)\times I_\sigma(t)} 
     \!\!\!\!\!\!
       |u|(y',s) \,  dy'ds\,
   +C \sigma^{-2}
      \!\!\!\!\!
     \iint  \limits_{B(x',\sigma)\times I_\sigma(t)}
     \!\!\!\!\!\!
     \frac{|D^k u(y',s)|}{|x'-y'|^{n-k+j}}  \, dy'ds \\
        &\leq C \sigma^{-j}M_{\beta,\text{loc}}u(x',t)+ C\sigma^{-2}
        \!\!\!\!\!
           \iint \limits_{{B(x',\sigma)\times I_\sigma(t)}}
        \!\!\!\!\!\!   
             \frac{|D^k u(y',s)|}{|x'-y'|^{n-k+j}}dy'ds
\end{align*}
for all $x= (x',t)\in P$ and $\sqrt{2}\sigma\leq r_P$.

   As for the second term, splitting the integral dyadically, we obtain
that is bounded by
\begin{equation}\label{MaximalD}
\sigma^{k-j}\sum_{i=0}^\infty 2^{i(j-k)}\,\frac{1}{\sigma^2|2^{-i}B(x'\sigma) |}\int_{I_\sigma(t)}\int_{2^{-i}B(x',\sigma)}|D^k u(y',s)|dy'ds.
\end{equation}
Since for $x\in P$ and $\sqrt{2}\sigma\leq r_P$ all averages involved correspond to balls in $\mathcal{F}_{1/2}$ and $j<k$, the term in \eqref{MaximalD}
is bounded by a constant times $\sigma^{k-j}M_{\beta,\text{loc}}D^ku (x)$ for all $x\in P$.

Putting together both estimates and taking $\sqrt{2}\sigma=\varepsilon r_P$, using that $r_P\simeq \delta(x)$ for $x\in P$ and denoting
\begin{equation*}
M_{\text{loc}}^2f(x',t)
   =\sup_{\substack{s\in I_\sigma(t)\\   \sigma\leq r_P }}
   \frac{1}{\sigma^2}\int_{I_\sigma(t)}|f(x',s)|\; ds,
\end{equation*}
we obtain
\begin{align}\label{cotamloc}
|D^\gamma (u)(x',t)|
   &\leq C
M_{\text{loc}}^2 (D^\gamma u)(x',t) \\
 \notag  
 &  \leq C\big((\varepsilon\delta(x))^{-j} 
 M_{\beta,\text{loc}}(u)(x) 
   +  (\varepsilon \delta(x))^{k-j}M_{\beta,\text{loc}}(D^{k}u(x)\big)
\end{align}
for a.e. $ (x',t) \in P$. Since $\mathcal W_{r_0 }$ is a covering of $\Omega_T$  and the right hand side of \eqref{cotamloc} no longer depends of $P$, we obtain that
(\ref{cotamloc}) holds for a.e. $x=(x',t)\in\Omega_T$.

Multiplying both sides by $\delta^j(x)$ and taking the norm in $L^p_w(\Omega_T)$, we arrive to
\begin{equation*}
\|\delta^j\,D^\gamma u \|_ {L^p_w(\Omega_T)}\leq C\bigl(\varepsilon^{-j}\|M_{\beta,\text{loc}}u\|_{L_w^p(\Omega_T)}+
\varepsilon^{k-j}\|M_{\beta,\text{loc}}(D^{k} u)\|_{L^p_{w\delta^{kp}}(\Omega_T)}\bigr).
\end{equation*}
Next, we observe that if the weight $w$ belongs to $A_{p,\text{loc}}(\Omega_T)$ also does $w\delta^s$, for any real number $s$. In fact, for any ball $B$ in $\mathcal{F}_{1/2}$
we have that $\delta(x)\simeq\delta(x_B)$, for any $x\in B$ so that  \eqref{apbeta} holds provided it is satisfied by $w$.

Therefore, an application of the continuity results for $M_{\beta,\text{loc}}f$, given in Theorem \ref{thm:acot.mxml.loc}, leads to the interpolation inequality \eqref{interin}.

\end{proof}

   Next we state the parabolic version of Theorem \ref{thm:bound.D2u}.

   \begin{thm}[See \cite{BC} and \cite{PS}]
      \label{thm:bound.D2uP}
        Under assumptions (1) and (2), for any   $p\in(1,\infty)$ and
          $w\in A_{p,\text{loc}}(\Omega_T )$,
           there exist  $C$ and $r_0> 0$
           such that for any ball $B_{0}= B(z_0, r_0)$  in  $\Omega_T$ with $10 B_{0}\in\mathcal{F}_{\beta}$ and any  $u\in W^{2,p}_0(B_0)$ the following inequalities hold
		
    \begin{align*}
	 	\|u_{x_{i}x_{j}}\|_{L^{p}_{w}(B_{0})} 
	 	    & \leq C  \|A_{P}u\|_{L^{p}_{w}(B_{0})} , \\
	 	\|u_{t}\|_{L^{p}_{w}(B_{0})} 
	 	   & \leq C \|A_{P}u\|_{L^{p}_{w}(B_{0})}.
 	\end{align*}

   \end{thm}
\begin{proof} The proof is similar to the elliptic case, as is proved in Corollary 2.13 in \cite{BC}, by using again
 expansion into spherical harmonics on the unit sphere, this time  in $\RR^{n+1}$.  After that, all is reduced  to obtain $L^p$- boundedness of a parabolic
Calder\'{o}n-Zygmund operator $T$ and its conmutator on a ball $B$ contained in $\Omega_T$ (see Theorems 2.12 and the representation formula (1.4) in this paper).
We can look at the operator $T$ and its conmutator $[T,b]$ acting on functions defined over the space of homogeneous type $B$ equipped with the parabolic metric and
 the restriction of Lebesgue measure. As before, the weight $w\chi_{B}$ is in $A_p(B)$. By the weighted theory of singular integrals and conmutators  on  spaces of homogeneous type,
 (see again \cite{PS}), applied to our operators the result follows.
\end{proof}

 Now we focus our attention in the proofs of the main Theorem of this section, that is, the parabolic version of Theorem \ref{thm:potencial}.

 \begin{thm} \label{thm:potencialP}
 Let $a_{ij}\in VMO(\RR^{n+1}) $, for $i,j=1,\dots,n$, $V\in RH_q(\RR^{n}) $ with $1< p\leq q$, and $w\in A_{{{q-1}\over{q-p}}p,\text{loc}}(\Omega_T)$.
 Then there exist positive constants $C$ and $r_{0}$   such that for any
 ball $B_{0}= B(z_0, r_0)$  in  $\Omega_T$ with $10 B_{0}\in\mathcal{F}_{\beta}$ and any  $u\in C^{\infty}_{0}(B_0)$, we have that

\begin{equation*}
\|Vu\|_{{L^p_w}(B_0)}  \leq C\big\|Lu\|_{L^p_w(B_0)}.
\end{equation*}

\end{thm}
\begin{proof}

   For $z_{0}=(z'_{0},\tau)\in\Omega_{T}$ pick a ball $B_{0}:=B(z_{0},r_{0})$ with $r_{0}$ to be chosen later. Again we let $x_{0}\in B_{0}$ and fix the coefficients $a_{ij}(x_{0})$ to obtain the operator

\begin{equation*}
L_{0}u=u_{t}-\sum _{i,j=1}^{n} a_{ij}(x_{0})u_{x_{i}}u_{x_{j}} + Vu = A_{0}u+Vu.
\end{equation*}

   From \cite{K} we know that the fundamental solution for this operator is bounded by the expression (see section \ref{prelim:fundamentalsolutions}):
\begin{align*}
|\Gamma(x_{0},x,y)|  &\leq C_{k}{1\over{\big(1+{{d(x,y)}\over{\rho(x')}}\big)^{k}}}  {1\over{d(x,y)^{n}}},
\end{align*}
for every $x=(x',t), y=(y',s)\in\Omega_{T}$, $t>s$, $k>0$, and for some constants $C_{k},C_{0}$ independent of $x_0$. Here again $\rho(x')$ is the critical radious.

   As usual, we defreeze the coefficients to obtain \ref{u.defreeze} and again the following pointwise bound holds for all $k\in\NN$, $x\in B_{0}$,

\begin{align}
\label{VuP}
|V(x')u(x)| 
  & \leq C_{k} V(x') \int  _{B_{0}} {1\over {\big( 1+{{d(x,y)}\over{\rho(x')}} \big)^{k} }} {1\over{d(x,y)^{n}}} \bigg( |Lu(y)|  +  \\
  \notag  
  &  \qquad\qquad\qquad\qquad
   + \sum _{i,j=1}^{n} |a_{ij}(y)-a_{ij}(x)|
 |u_{x_{i}x_{j}}(y)| \bigg) dy,
\end{align}
and rewrite (\ref{VuP}) as
\begin{align}
\label{9P}
|V(x')u(x)| \leq C_{k} S_{k}(|Lu|)(x)+\sum _{i,j=1}^{n} S_{k,a_{ij}}(|u_{x_{i}x_{j}}|)(x),
\end{align}
where $S_{k}$ and $S_{k,a}$ are the integral operators defined as
\begin{align*}
S_{k}f(x)  &= V(x') \int  {1\over {\big( 1+{{d(x,y)}\over{\rho(x')}} \big)^{k} }} {1\over{d(x,y)^{n}}} f(y) dy, \qquad \mbox{ and} \\
S_{k,a}f(x)& =  V(x') \int  {1\over {\big( 1+{{d(x,y)}\over{\rho(x')}} \big)^{k} }} {1\over{d(x,y)^{n}}} |a(y)-a(x)| f(y) dy,
\end{align*}
with $a\in L^{\infty}\cap VMO(\RR^{n})$, $k\in\NN$.

   Thus, as in the elliptic case, the  Theorem  follows from Theorem \ref{thm:bound.D2uP} and the next parabolic version of Theorems \ref{thm:Sk.sin.conmutador} and \ref{thm:Ska.con.conmutador}.

\end{proof}

   Now we need to prove the following parabolic version of Theorem \ref{thm:Sk.sin.conmutador}:

\begin{thm}
    \label{thm:Sk.sin.conmutadorP}
   Let $B_{0}$ be a ball in $\mathcal{F}_{\beta}$
  such that $10 B_{0}\in\mathcal{F}_{\beta}$.
  Then for $k$ large enough and $p \in [1,q]$,
  the operator $S_{k}$ is bounded on $L^{p}_{w}(B_{r_{0}})$, with $w\in A_{{{q-1}\over{q-p}}p ,\text{loc}}(\Omega_{T})$.
  \end{thm}

\begin{proof}
   This proof is also done by duality. The remarks we made along the proof of Theorem \ref{thm:Sk.sin.conmutador} also hold this time so we won't mention them.

   The adjoint operator of $S_{k}$ is
\begin{align*}
S^{\ast}_{k}f(x)&= \int  {{V(y')}\over{\big( 1+ {{d(x,y)}\over{\rho(y')}}\big)^{k}}} {1\over{d(x,y)^{n}}} f(y)dy, \qquad x\in\Omega_{T}.
\end{align*}

   Just like before we can split
\begin{align*}
S^{\ast}_{k}f(x) & \leq  C \int _{d(x,y)<\rho(x')} {1\over{d(x,y)^{n}}}  V(y') \chi_{B_{0}}(y)f(y) dy \, + \\
& \qquad \qquad 
    + C \int _{d(x,y)\ge\rho(x')} \Big({{\rho(x')}\over{d(x,y)}} \Big)^{k} {1\over{d(x,y)^{n}}} V(y') \chi_{B_{0}}(y)f(y) dy \\
    & = \mathbf{A}(x)+\mathbf{B}(x).
\end{align*}
We will prove the pointwise  {bound}
\begin{equation*}
S^{\ast}_{k}f(x)\leq C M_{q',\text{loc}}(f)(x).
\end{equation*}

   In order to study $\mathbf{A}(x)$, let $x\in B_{0}=B(z_{0},r_{0})$. Denote by $B_{j}$ the balls $B_{j}=B(x,2^{-j}\rho(x'))$, by $C_{j}$ the annuli defined as $C_{j}=\{ y: 2^{-(j+1)}\rho(x')<d(x,y)\leq 2^{-j}\rho(x) \} = \overline{B_{j}}\backslash\overline{B_{j+1}}$, and by $R_{j}$ the rectangles $R_{j}=B'_{j}\times I_{j}$ where $B'_{j}$ denotes the ball in $\RR^{n}$, $B'_{j}=B(x',2^{-j}\rho(x'))$ and $I_{j}$ denotes the real ball $I_{j}=B(t,(2^{-j}\rho(x')^{2})$, $j\in\NN_{0}$. We have that $C_{j}\subset B_{j} \subset R_{j}$, and let us remark that the ball measures are $|B_{j}|= c_n(2^{-j}\rho(x))^{n+2}$ and $|B'_{j}|=C_n(2^{-j}\rho(x'))^{n}$. The same steps as before prove that
\begin{equation*}
\mathbf{A}(x)\leq C M_{q',\text{loc}}(f)(x),
\end{equation*}
  for $x\in B_{0}$, $f\in L^{p}_{w}(B_{0})$ and $f\ge 0$, where $M_{q',\text{loc}}$ denotes the local maximal function of exponent $q'$, in the parabolic setting. Indeed, if $\rho(x')\leq r_{0}$ we have that
\begin{align*}
\mathbf{A}(x)& \leq  C \sum _{j=0}^{\infty} {{|B_{j}|}\over{(2^{-j}\rho(x'))^{n}}} \bigg( {1\over{|B'_{j}|}} \int _{B'_{j}} V(y')^{q}dy' \bigg)^{1\over q} \bigg( {1\over{|B_{j}|}} \int _{B_{j}} f(y)^{q'}dy \bigg)^{1\over{{q'}}}  \\
&\leq  C M_{q',\text{loc}}(f)(x) \sum _{j=0}^{\infty} (2^{-j}\rho(x'))^{2} \bigg({1\over{|B'_{j}|}}\int _{B'_{j}} V(y')dy'\bigg),
\end{align*}
because of the H\"{o}lder inequality, the reverse H\"{o}lder condition $V$ and the definition of local maximal function. And in the case $\rho(x') >r_{0}$, again there exists $j_{0}\in\NN_{0}$ such that $C_{j}\cap B_{0}=\emptyset$ for $j\leq j_{0}+2$. The same steps as before show us that
\begin{equation*}
\mathbf{A}(x)
  \leq C M_{q',\text{loc}}(f)(x) \sum _{j=j_{0}-1}^{\infty} (2^{-j}\rho(x'))^{2} \bigg({1\over{|B'_{j}|}}\int _{B'_{j}} V(y')dy'\bigg).
\end{equation*}
Now we use again equations \eqref{13} and \eqref{14} to conclude that $\mathbf{A}(x)\leq C M_{q',\text{loc}}(f)(x)$.

   To study $\mathbf{B}(x)$, we consider the balls $B_{j}=B(x,2^{j}\rho(x'))$, the annuli $C_{j}=\{y:2^{j}\rho(x')<d(x,y)\leq 2^{j+1}\rho(x')\}$, and the rectangles $R_{j}=B'_{j}\times I_{j}=B(x',2^{j}\rho(x')) \times B(t,(2^{j}\rho(x'))^{2}) \subset \RR^{n}\times \RR$, for $j\in\NN_{0}$. We have that $C_{j}\subset B_{j} \subset R_{j}$. Observe that if $\rho(x')>2r_{0}$, then $\mathbf{B}(x)=0$, thus we consider only the case $\rho(x')\leq 2r_{0}$. There exists $j_{0}\in\NN_{0}$ such that $C_{j}\cap B_{0}=\emptyset$ if $j\ge j_{0}+2$. Thus we have that
\begin{align*}
\mathbf{B}(x)& \leq  C M_{q',\text{loc}}(f)(x) \sum _{j=0}^{j_{0}+1} {{(2^{j}\rho(x'))^{2}}\over{2^{jk}}} \bigg({1\over{|B'_{j}|}}\int _{B'_{j}} V(y')dy'\bigg),
\end{align*}
because of the use of H\"{o}lder inequality, the reverse H\"{o}lder conditionon $V$ and the definition of local maximal function of the order $q'$. Thus, using again equations \ref{13} and \ref{14}, $\mathbf{B}(x)\leq C M_{q',\text{loc}}(f)(x)$.

\end{proof}

\begin{Rema}\label{thm:Sk.sin.conmutadorPsin pesos}
We note that arguing in a similar way as in the proof of Theorem \ref{thm:Sk.sin.conmutadorP} it can be show that the operator $S_{k}$ is bounded on
$L^{p}(\RR^{(n+1)})$ with $w = 1$ and $p\in [1,q]$. In this case the operator is pointwisely
bounded by the maximal Hardy-Littlewood function of order $q'$.
\end{Rema}
   We turn now to the proof of parabolic Theorem \ref{thm:Ska.con.conmutador}:
\begin{thm}
    \label{thm:Ska.con.conmutadorP}
   Let $p \in (1,q]$ and $w\in A_{{{q-1}\over{q-p}}p,\text{loc}}(\Omega_T)$. Then, given $\epsilon >0$ there exist $r_{0}>0$, depending on the $VMO-$modulus of $a$
   such that for any ball $B_{0}= B(z_0, r_0)$  in  $\Omega_T$ with $10 B_{0}\in\mathcal{F}_{\beta}$ , the inequality
\begin{align}
\|S_{k,a}f\|_{L^{p}_{w}(B_{0})} \leq \epsilon \|f\|_{L^{p}_{w}(B_{0})}.
\end{align}
holds for all $f\in L^{p}_{w}(B_{0})$ and $k$ large enough.
\end{thm}

\begin{proof}
  This proof is also done by duality as in the proof of Theorem \ref{thm:Ska.con.conmutador}, and follows by Theorem \ref{thm:abstracto} with $\Lambda= \Omega_T$ and $b\in BMO(\RR^{n+1})$.

Here,
\begin{align*}
S^{\ast}_{k,a}f(x) &=  \int {{V(y')}\over{\big( 1+ {{d(x,y)}\over{\rho(y')}}\big)^{k}}} {1\over{d(x,y)^{n}}} |a(y)-a(x)| f(y)dy,
\end{align*}
for each positive integer $k$ and $a\in VMO$; and the kernel is
\begin{equation*}
w(x,y)={1\over {\big( 1+{{d(x,y)}\over{\rho(x')}} \big)^{k} }} {1\over{d(x,y)^{n}}},
\end{equation*}
 which satisfies the $H_{1}(q)$ condition as shown in section \ref{prelim:lemmas} (Lemma\ref{lemm:parabolickernelh1q})

\end{proof}

\section{Proof of the Main Result}
\label{mains}

We are in position to proof Theorem \ref{thm:principal}.

\begin{proof}[Proof of Theorem \ref{thm:principal}]
  Let   $\mathcal W_{r_{0}}=\{B_{i}=B(x_{i},r_{i})\}$ be a covering as in Lemma \ref{lemm:covering.Omega}, with $r_{0}$ as in Theorems \ref{thm:bound.D2u} and \ref{thm:potencial}
       and $0<r_0 < \beta /10$. For each $B_i\in  W_{r_0 }$ we consider a function $\eta_{i}$ such that the family $\{\eta_i\}_{i=1}^\infty$ satisfies

 \begin{enumerate}
           \item $\eta_i\in\mathcal C_0^\infty(2 B(x_i,r_i))$, $\eta_i\equiv 1$ in $B_i$,
           \item $\|\eta_i\|_\infty\leq 1$, $\|D^\alpha \eta_i\|_\infty\leq Cr_i^{-|\alpha |}$ where $ r_i\approx d(x_i, \partial\Omega )$ if $B(x_i,r_i)\in \tilde{\mathcal G}_{r_0 }$ and $  r_i \approx 1$ when $B(x_i,r_i)\in \mathcal G_{r_0 }$,
           \item $\sum _{i=1}^\infty \chi_{2B_i}(x)\leq M$.
\end{enumerate}
 By using Theorem \ref{thm:bound.D2u}, for each $i$, we get
\begin{align*}
  \| \chi_{B_i} D^2    &    (u\eta_{i})  \|^p_{L^p_w(2B_i)}\\
         & \leq C \| A(u\eta_i)\|^p_{L^p_w(2 B_i)}\\
         & \leq C \big(\| Au\|^p_{L^p_w(2B_i)}
            + r_i^{-1}\| Du\|^p_{L^p_w(2B_i)}
            + r_i^{-2}\|u\|^p_{L^p_w(B_i)} \big)^{p}\\
         & \leq C \big(\| Au\|_{L^p_w(2B_i )}
                + r_i^{-1}\| Du\|_{L^p_w(2B_i  )}
                +  r_i^{-2}\|u\|_{L^p(2B_i )}\big)^{p} \\
         & \leq C\big(\| Lu\|_{L^p_w( 2B_i)}
              +\| Vu\|_{L^p_w(2B_i )}
              + r_i^{-1}\|Du\|_{L^p_w(2B_i )}
              + r_i^{-2}\|u\|_{L^p_w(2B_i )}\big)^p.
  \end{align*}

Analogously, using this time Theorem \ref{thm:potencial}, since $w\in A_{p,\text{loc}}(\Omega )\subset A_{\frac{q-1}{q-p}p,\text{loc}}(\Omega) $  we obtain
\begin{align*}
\| \chi_{B_i} V (u\eta_{i})\|^p_{L^p_w(2B_i )}
    &\leq C\|L(u\eta_{i})\|^p_{L^p_w(2B_i)} \\
    &\leq C \|Lu\|^p_{L^p_w(B_i)}+r_i^{-1} \|Du\|^p_{L^p_w(B_i)}+r_i^{-2} \|u\|^p_{L^p_w(B_i)}.
 \end{align*}

Now, we note that for $x\in B_i$ the function $\eta_i u$ coincides with $u$, and also for $x\in 2B_i$, we have
$\delta (x_i)\approx r_i $ with  $\delta (x_i)=\min\{1,d(x_i,\boundary \Omega )\}$. Hence, putting together both estimates, multiplying both sides  by $\delta^2$, adding over $i$, using de finite overlapping property of the covering $\{2B_i \}$ and taking the $1/p$-th power, we arrive to

 \begin{align*}
     \|u\|_{W^{2,p}_{\delta ,w}(\Omega )}+\|\delta ^2Vu\|_{L^p_w(\Omega )}
     \leq C(\|\delta ^2 Lu\|_{L^p_w(\Omega )} + \|\delta Du\|_{L^p_w(\Omega  )} + \|u\|_{L^p_w(\Omega )})
     .\\
    \intertext{Using the interpolation Theorem \ref{thm:bound.epsilon}}
       \leq C(\|\delta ^2 Lu\|_{L^p_w(\Omega  )} + \epsilon \|\delta^2 D^2 u\|_{L^p_w(\Omega )})+(C+\epsilon^{-1}) \|u\|_{L^p_w(\Omega  )}.
    \end{align*}

  Finally, choosing $\epsilon $ such that $C\epsilon = 1/2$ and subtracting the term  $\|\delta ^2 D^2 u\|_{L^p_w(\Omega )}$, it follows
    \begin{equation*}
      \|u\|_{W^{2,p}_{\delta ,w}(\Omega  )}
                  \leq C \{\|Lu\|_{L^p_w(\Omega  )}+\|u\|_{L^p_w(\Omega )}\},
    \end{equation*}
whence the desired estimate follows.

  \end{proof}

   The proof of Theorem \ref{thm:principalP} is obtained by a few changes:

\begin{proof}[Proof of Theorem \ref{thm:principalP}]
  Just like in the previous proof, from Lemma \ref{lemm:covering.Omega} applying this time to $\Gamma= \Omega_T$, we consider a covering $\mathcal{W}_{r_{0}}$ and a family $\{\eta_{i}\}$ which satisfies 1 and 3, and the following 2: $\|\eta_{i}\|_{\infty}\leq 1$,
  \begin{align*} \|D^{\alpha}_{x}\eta_{i}\|_{\infty} \leq C r_{i}^{-|\alpha|}, \\ \|D_{t}\eta_{i}\|_{\infty}\leq C r_{i}^{-2},
 \end{align*}
 where $ r_i\approx d(x_i, \partial\Omega )$ if $B(x_i,r_i)\in \tilde{\mathcal G}_{r_0 }$ and $  r_i \approx 1$ when $B(x_i,r_i)\in \mathcal G_{r_0 }$.

  Now for each $i$ we use  theorems 
  \ref{thm:bound.D2uP} and \ref{thm:potencialP} to get
\begin{align*}
\|\chi _{B_{i}}D_{x}^{2}(u\eta_{i})\|_{L_{w}^{p}(2B_{i})} 
& \leq C \big( \| Lu\|_{L^p_w( 2B_i)}
    +\| Vu\|_{L^p_w(2B_i )}  + \\
    & \qquad\qquad\qquad
    + r_i^{-1}\|Du\|_{L^p_w(2B_i )}
    + r_i^{-2}\|u\|_{L^p_w(2B_i )}\big), \\
\|\chi _{B_{i}}D_{t}(u\eta_{i})\|^{p}_{L_{w}^{p}(2B_{i})}  
    &\leq C \big( \|Lu\|_{L^p_w( 2B_i)}
       +\| Vu\|_{L^p_w(2B_i )} + \\
     & \qquad\qquad\qquad
     + r_i^{-1}\|Du\|_{L^p_w(2B_i )}+ r_i^{-2}\|u\|_{L^p_w(2B_i )} \big), \\
\|\chi _{B_{i}}Vu\eta_{i}\|^{p}_{L_{w}^{p}(2B_{i})} & \leq C \big( \|Lu\|_{L_{w}^{p}(2B_{i})} + r_{i}^{-1}\|D_{x}u\|_{L_{w}^{p}(2B_{i})}+ r_{i}^{-2}\|u\|_{L_{w}^{p}(2B_{i})}\big),
\end{align*}
then, by performing analogous operations to the previous Theorem, we obtain
			\begin{align*}
			\|u\|_{W^{2,p}_{\delta ,w} (\Omega_{T})} + \|\delta^{2}Vu\|_{L^{p}_{w}(\Omega_{T})}  &\leq C\big( \|\delta^{2}Lu\|_{L^{p}_{w}(\Omega_{T})} + \|\delta D_{x}u\|_{L^{p}_{w}(\Omega_{T})} +\|u\|_{L^{p}_{w}(\Omega_{T})} \big).
			\end{align*}
From the interpolation Theorem \ref{thm:bound.epsilonP} we have that
			\begin{align*}
						\|\delta D_{x}u\|_{L^{p}_{w}(\Omega_{T})}
					 & \leq C \big( \epsilon^{-1}\|u\|_{L^{p}_{w}(\Omega_{T})} + \epsilon \|\delta^{2}D^{2}_{x}u\|_{L^{p}_{w}(\Omega_{T})} \big),
			\end{align*}
which finally leads us to
			\begin{align*}
					\|u\|_{W^{2,p}_{\delta ,w} (\Omega_{T})} + \|\delta^{2}Vu\|_{L^{p}_{w}(\Omega_{T})} & \leq C\big( \|\delta ^2 Lu\|_{L^p_w(\Omega_{T} )}+\|u\|_{L^p_w(\Omega_{T} )}\big)
			\end{align*}
as we desired.
\end{proof}

\end{document}